\documentclass[11pt]{article}
\usepackage{amsthm,amstext, amsmath,latexsym,amsbsy,amssymb,cases}

\newtheorem{cor}{Corollary}[section]
\newtheorem{definition}[cor]{Definition}
\newtheorem{cl}[cor]{Claim}
\newtheorem{lem}[cor]{Lemma}
\newtheorem{prop}[cor]{Proposition}

\newcommand\g{G}
\newcommand\rr{R}
\newcommand\bb{\delta}
\newcommand\atang{{\rm arctanh}}
\newtheorem{theorr}{Theorem}
\newtheorem{propp}[theorr]{Proposition}

\title{\bf Blow-up profile for the complex-valued semilinear wave equation}
\author{Asma Azaiez\\
\small Universit\'e Paris 13, Sorbonne Paris Cit\'e \\
\small LAGA, CNRS (UMR 7539)\\
\small F-93430, Villetaneuse - France}
\begin{document}
\maketitle 
\begin{abstract} In this paper, we consider a blow-up solution for the complex-valued semilinear wave equation with power nonlinearity in one space dimension. We first characterize all the solutions of the associated stationary problem as a two-parameter family. Then, we use a dynamical system formulation to show that the solution in self-similar variables approaches some particular stationary one in the energy norm, in the non-characteristic case. This gives the blow-up profile for the original equation in the non-characteristic case. Our analysis is not just a simple adaptation of the already handled real case. In particular, there is one more neutral-direction in our problem, which we control thanks to a modulation technique.
 \end{abstract}
complex valued PDE.

\begin{flushright}
{\it Paper accepted for publication in Trans. Amer. Math. Soc.\;\;\;\;\;\;}
\end{flushright}
\bigskip
\section{Introduction}
\subsection{The problem and known results}
We consider the following complex-valued one-dimensional semilinear wave equation
\begin{equation} \left\{
\begin{array}{l}
\displaystyle\partial^2_{t} u = \partial^2_{x} u+|u|^{p-1}u, \\
u(0)=u_{0} \mbox{ and }  u_{t}(0) = u_{1},
\end{array}
\right . \label{waveq}
\end{equation}
where $u(t):  x\in \mathbb{R}\to u(x,t) \in \mathbb{C} ,\, p>1,\, u_0 \in
H^1_{loc,u}$ and $ u_1\in L^2_{loc,u}$, with $$||
v||^2_{L^2_{loc,u}}=\displaystyle\sup\limits_{a\in \mathbb{ R}}
\int_{|x-a|<1}|v(x)|^2 dx  \mbox{ and }|| v||^2_{H^1_{loc,u}}=||
v||^2_{L^2_{loc,u}}+||  \nabla v||^2_{L^2_{loc,u}}\cdot$$

\medskip
\noindent The Cauchy problem for equation (\ref{waveq}) in the space $H^1_{loc,u}\times L^2_{loc,u}$ follows from the finite speed of propagation and the wellposedness in $H^1\times L^2$. See for instance Ginibre, Soffer and Velo \cite{MR1190421}, Ginibre and Velo \cite{MR1256167}, Lindblad and Sogge \cite{MR1335386} (for the local in time wellposedness in $H^1\times L^2$). The existence of blow-up solutions for equation (\ref{waveq}) is a consequence of the finite speed of propagation and ODE techniques (see for example Levine \cite{l74} and Antonini and Merle \cite{MR1861514}). More blow-up results can be found in Caffarelli and Friedman \cite{MR849476}, Alinhac \cite{ali95} and \cite{MR1968197}, Kichenassamy and Littman \cite{KL93a}, \cite{KL93b} Shatah and Struwe \cite{MR1674843}.
%  More recently, in \cite{MR2847755} Nakanishi and Schlag give recent advances on the blow-up problem for equation (\ref{waveq}) and related PDEs.
% 
% \centerline{***********}

The real case (in one space dimension) has been understood completely, in a series of papers by Merle and Zaag \cite{MR2362418}, \cite{MR2415473}, \cite{MR2931219} and \cite{Mz12} and in C\^ote and Zaag \cite{CZ12} (see also the note \cite{Mz10}). Some of those results have been extended to higher dimensions for conformal or subconformal $p$:
\begin{equation}\label{conp}
 1<p\le p_c \equiv 1+\frac{4}{N-1},
\end{equation}
under radial symmetry outside the origin in \cite{MR2799813}. For non radial solutions, we would like to mention \cite{MZ05} and \cite{MZ2005} where the blow-up rate was obtained. We also mention the recent contribution of \cite{MZ13} and \cite{MZ13'} where the blow-up behavior is given, together with some stability results.

 Considering the behavior of radial solutions at the origin, Donninger and Sch{\"o}rkhuber \cite{MR2909934} were able to prove the stability of the space-independent solution (i.e. the solution of the associated ODE $u''=u^p$) with respect to perturbation in the initial data. Willing to be as exhaustive as possible in our bibliography about the blow-up question for equation (\ref{waveq}), we would like to mention some blow-up results in the superconformal, Sobolev critical and supercritical ranges for equation (\ref{waveq}).

When 
$$N\ge 2\mbox{ and }p_c<p<p_s\equiv\frac{N+2}{N-2},$$
Killip, Stoval and Visan found in \cite{MR2869186} an upper bound on the blow-up rate. That bound is larger than the solution of the associated ODE $u''=u^p$, and is therefore thought to be non optimal. In \cite{MR3130348} Hamza and Zaag gives a different proof of the results of \cite{MR2869186}, improving some of their estimates.

When
$$N\ge 3\mbox{ and }p=p_s,$$
equation (\ref{waveq}) has attracted a lot of interest. Many authors addressed the question of obtaining sufficient conditions for scattering and blow-up through energy estimates, in relation with the ground state (see Kenig and Merle \cite{20}, Duyckaerts and Merle \cite{14}). Furthermore, dynamics around the soliton were studied: see Krieger and Schlag \cite{26}, Krieger, Nakanishi and Schlag \cite{24} and \cite{25}. There are also some remarkable classification theorems by Duyckaerts, Kenig and Merle \cite{10}, \cite{13}, \cite{12} and \cite{11}. Concerning the blow-up behavior, we would like to mention that Donninger, Huang, Krieger and Schlag prove in \cite{6} the existence of so-called ``exotic'' blow-up solutions when $N=3$, whose blow-up rate oscillates between several pure-power laws.

When 
$$N\ge 3\mbox{ and }p>p_s,$$
much less is known. We would like just to mention that the stability result of Donninger and Sch\"{o}rkhuber proved in \cite{ds} the superconformal range, does hold in the Sobolev supercritical range too, at least when $N=3$. There is also a remarkable result by Kenig and Merle \cite{21} on the dynamics of solutions with some compactness property.

\medskip
 In this work, our aim is to study the profile of blow-up solutions in the complex case of equation (\ref{waveq}). In particular, relying on the work of Merle and Zaag in \cite{MR2362418}, we give a trapping result near the set of non-zero stationary solutions in self-similar variables. This study is far from being trivial since the complex structure introduces an additional zero eigenfunction in the linearized equation around the expected profile, and also because of the coupling between the real and the imaginary parts.

\medskip
If \emph{u} is a blow-up solution of (\ref{waveq}), we define (see
for example Alinhac \cite{ali95}) a continuous curve $\Gamma$ as the graph of a function ${x \mapsto T(x)}$ such that the domain of definition of $u$ (or the maximal influence domain of $u$) is 
\begin{equation*}\label{domaine-de-definition}
D_u=\{(x,t)| t<T(x)\}.
\end{equation*}

\noindent From the finite speed of propagation, $T$ is a 1-Lipschitz function. The time $\check {T}=\inf_{x \in \mathbb R}T(x)$ and the graph $\Gamma$  are called (respectively) the blow-up time and the blow-up graph of $u$.

\noindent Let us introduce the following non-degeneracy condition for $\Gamma$. If we introduce for all $x \in \mathbb{R},$ $t\le T(x)$ and $\delta>0$, the cone
\begin{equation*}
 \mathcal{C}_{x,t, \delta }=\{(\xi,\tau)\neq (x,t)\,| 0\le \tau\le t-\delta |\xi-x|\},
\end{equation*}
then our non-degeneracy condition is the following: $x_0$ is a non-characteristic point if
\begin{equation}\label{4}
 \exists \delta_0 = \delta (x_0) \in (0,1) \mbox{ such that } u \mbox{ is defined on }\mathcal{C}_{x_0,T(x_0), \delta_0}.
\end{equation}
If condition (\ref{4}) is not true, then we call $x_0$ a characteristic point. Already when $u$ is real-valued, we know from \cite{MR2931219} and \cite{CZ12} that there exist blow-up solutions with characteristic points.

\medskip
 Given some $x_0 \in \mathbb R,$ we introduce the following self-similar change of variables:
\begin{equation}\label{trans_auto}
w_{x_0}(y,s) =(T(x_0)-t)^\frac{2}{p-1}u(x,t), \quad  y=\frac{x-x_0}{T(x_0)-t}, \quad
s=-\log(T(x_0)-t).
\end{equation}
This change of variables transforms the backward light cone with vertex $(x_0, T(x_0))$ into the infinite cylinder $(y,s)\in (-1,1) \times [-\log T(x_0),+\infty).$ The function $w_{x_0}$ (we write $w$ for simplicity) satisfies the following equation for all $|y|<1$ and $s\ge -\log T(x_0)$:
\begin{eqnarray} \label{equa}
 \partial^2_{s} w=\mathcal{L}w-\frac{2(p+1)}{(p-1)^2}w+|w|^{p-1}w-\frac{p+3}{p-1} \partial_s w- 2 y \partial_{ys} w
 \end{eqnarray}
 \begin{eqnarray}\mbox{where }\mathcal{L} w=\frac{1}{\rho}\partial_y (\rho (1-y^2)\partial_y w)\, \mbox{ and }\, \rho (y)= (1-y^2)^\frac{2}{p-1}.\label{8}
\end{eqnarray}
This equation will be studied in the space
 \begin{eqnarray}\label{9}
\mathcal{ H}=\{ q \in H_{loc}^1 \times L_{loc}^2 ((-1,1),\mathbb{C}) \, \Big| \parallel  q \parallel_{\mathcal{ H}}^2 \equiv \int_{-1}^{1}(|q_1|^2+|q'_1|^2(1-y^2)+|q_2|^2)\rho\;dy< +\infty \} ,
 \end{eqnarray}
which is the energy space for $w$. Note that $\mathcal{ H}=\mathcal{ H}_0\times L_{\rho}^2$ where
 \begin{eqnarray}\label{10}
\mathcal{ H}_0=\{ r \in H_{loc}^1  ((-1,1),\mathbb{C}) \,\Big| \parallel  r \parallel_{\mathcal{ H}_0}^2 \equiv \int_{-1}^{1}(|r'|^2(1-y^2)+|r|^2)\rho\;dy< +\infty\} .
 \end{eqnarray}
Let us define 
\begin{equation}\label{15}
 E(w,\partial_s w)=\int_{-1}^{1} \left( \frac{1}{2} |\partial_s w|^2+\frac{1}{2} |\partial_y w|^2 (1-y^2)+\frac{p+1}{(p-1)^2}|w|^2-\frac{1}{p+1}|w|^{p+1}\right) \rho dy.
\end{equation}
By the argument of Antonini and Merle \cite{MR1861514}, which works straightforwardly in the complex case, we see that $E$ is a Lyapunov functional for equation (\ref{equa}).
Similarly, some arguments of the real case, can be adapted with no problems to the complex-case, others don't. As a matter of fact, the derivation of the blow-up rate works as in the real case whereas the convergence to the profile needs intricate estimates, and this is the goal of our paper. Let us first briefly state the result for the blow-up rate, then focus on the convergence question.
\subsection{Blow-up rate}

Only in this subsection, the space dimension will be extended to any $N\ge 1$. We assume in addition that $p$ is conformal or sub-conformal (see (\ref{conp})). We recall that for the real case of equation (\ref{waveq}), Merle and Zaag determined in \cite{MZ05} and \cite{MZ2005} the blow-up rate for (\ref{waveq}) in the region $\{(x,t)\,|\, t<\check T\}$ in a first step. Then in \cite{MR2147056}, they extended their result to the whole domain of definition $\{(x,t)\,|\, t< T(x)\}$. In the following, we give the growth estimate near the blow-up surface for solutions of equation (\ref{waveq}).
\begin{propp}\label{Th}{\bf (Growth estimate near the blow-up surface for solutions of equation (\ref{waveq}))}
If $u$ is a solution of (\ref{waveq}) with blow-up surface $\Gamma \,:\, \{x\rightarrow T(x)\},$ and if $x_0\in \mathbb{R}^N$ is non-characteristic (in the sense (\ref{4})) then,
\item{(i)} {\bf (Uniform bounds on $w$)} For all $s\ge -\log \frac{T(x_0)}{4}$:
$$0 < \epsilon_0(N, p)\le||w_{x_0}(s)||_{H^1(B)}+||\partial_s w_{x_0}(s)||_{L^2(B)}\le K.$$
\item{(ii)} {\bf (Uniform bounds on $u$)} For all $t\in [\frac{3}{4}T(x_0),T(x_0))$:
\begin{align*}
 &0 < \epsilon_0(N, p)\le(T(x_0)-t)^\frac{2}{p-1}\frac{||u(t)||_{L^2(B(x_0,T(x_0)-t))}}{(T(x_0)-t)^{N/2}}\\
&+(T(x_0)-t)^{\frac{2}{p-1}+1}\left( \frac{||\partial_t u(t)||_{L^2(B(x_0,T(x_0)-t))}}{(T(x_0)-t)^{N/2}}+\frac{||\nabla u(t)||_{L^2(B(x_0,T(x_0)-t))}}{(T(x_0)-t)^{N/2}}\right) \le K,
\end{align*}
where the constant $K$ depends only on $N,\, p,$ and on an upper bound on $T(x_0),1/T(x_0)$, $\delta_0 (x_0)$ and the initial data in $H^1_{loc,u}\times L^2_{loc, u}$.
\end{propp}
\begin{proof}
The idea of the proof is the same as in \cite{MR2147056}. For the sake of completeness, we give in Appendix \ref{__} a sketch of the proof.
\end{proof}
With the bounds in Proposition \ref{Th}, we ask the question of compactness of the solution and the question of convergence of $w$ to a stationary solution of (\ref{equa}).
\subsection{Blow-up profile}
From now on, we assume again that
$$N=1.$$
This subsection is the heart of our work. Indeed, unlike for the blow-up rate, it is not a simple adaptation of the real case. It involves many new ideas of ours. The first step towards the determination of the blow-up profile is to characterize  all stationary solutions in $\mathcal{ H}_0$ of equation (\ref{equa}).
\begin{propp}\label{p1}{\bf (Characterization of all stationary solution of equation (\ref{equa}) in $\mathcal{ H}_0$).}
 (i) Consider $w \in \mathcal{H}_0$ a stationary solution of (\ref{equa}). Then, either $w\equiv 0$ or there exist $\bb \in(-1,1)$ and $\theta \in \mathbb{R}$ such that $w(y)=e^{i\theta} \kappa (\bb,y)$ where
\begin{eqnarray}\label{defk}
\forall (\bb, y) \in (-1,1)^2,\;\kappa(\bb,y)= \kappa_0 \frac{(1-\bb^2)^\frac{1}{p-1}}{(1+\bb y)^\frac{2}{p-1}} \mbox{ and }\kappa_0=\left(\frac{2(p+1)}{(p-1)^2}\right)^\frac{1}{p-1},
\end{eqnarray}
(ii) It holds that
\begin{equation}\label{14}
 E(0,0)=0 \mbox{ and }\forall \bb \in (-1,1),\,\forall \theta \in \mathbb R ,\, E(e^{i\theta}\kappa (\bb,\cdot),0)=E(\kappa_0,0)>0
\end{equation}
where $E$ is given by (\ref{15}).
\end{propp}
\noindent {\bf Remark:} Note that the proof of this proposition is very different from the real case. Indeed, in the real case, the result follows from a transformation of the hyperbolic plane, which gives nothing in the complex case. We succeed in proving the result relying on ODE techniques for complex-valued equation.\\
{\bf Remark:} Unlike the real case where the set of stationary solutions is made of 3 connected components: $\{0\}$, $\{+\kappa(\bb ,y)\}$ and $\{-\kappa(\bb ,y)\}$, we have only two connected components: $\{0\}$ and $\{e^{i\theta}\kappa(\bb ,y)\,|\, \theta \in \mathbb R, |\bb|<1\}$. This is one of the novelties of our approach. Indeed, we need here a modulation technique to control the parameter $\theta$ which may take any real value, unlike the real case, where it was equal to $k\pi$ only.

\medskip
The second step is the same as in the real case, and involves no novelty on our behalf. It uses the Lyapunov functional to show that when $x_0$ is non-characteristic, then $w_{x_0}$ approaches the set of non-zero stationary solutions. This is the result:
\begin{propp}\label{2}{\bf (Approaching the set of non-zero stationary solutions near a non-characteristic point)}
 Consider $u$ a solution of (\ref{waveq}) with blow-up curve $\Gamma :\{ x \rightarrow T(x)\}.$ If $x_0 \in \mathbb R$ is non-characteristic, then:
\item{(A.i)} $\inf_{\{\theta \in \mathbb R,\; |\bb|<1\}}||w_{x_0}(\cdot  ,s)-e^{i\theta} \kappa(\bb ,\cdot)||_{H^1(-1,1)}+||\partial_s w_{x_0}||_{L^2(-1,1)} \rightarrow  0$ as $s  \rightarrow \infty$.
\item{(A.ii)} $E(w_{x_0}(s),\partial_s w_{x_0}(s))\rightarrow E(\kappa_0,0)$ as $s \rightarrow \infty.$
\end{propp}
\begin{proof}
The idea of the proof is the same as in \cite{MR2362418}. For the sake of completeness, we give in Appendix \ref{__} a sketch of the proof.
\end{proof}
From this result, we wonder whether $\theta$ and $\bb$ have limits as $s\rightarrow \infty$, in this words, whether $w_{x_0}(\cdot,s)$ converges to some $e^{i\theta_\infty (x_0)}\kappa(\bb _\infty (x_0))$ for some $\theta_\infty (x_0)\in \mathbb{R}$ and $ |\bb_\infty (x_0)|<1$. The answer is in fact positive, as one sees in Theorem \ref{theo4} below that the following trapping result of solutions of equation (\ref{equa}) near non-zero stationary solutions, is a major tool towards this result.

 In the following, we consider $w \in C([s^*,\infty),\mathcal{ H})$ and show that if $w(s^*)$ is close enough to some non-zero stationary solution and satisfies an energy barrier, then $w(s)$ converges to a neighboring stationary solution as $s \rightarrow \infty$.
\begin{theorr}\label{theo3}{\bf (Trapping near the set of non-zero stationary solutions of (\ref{equa}))} There exist positive $\epsilon_0$, $\mu_0$ and $C_0$ such that if $w\in C([s^*,\infty),\mathcal{H})$ for some $s^*\in \mathbb{R}$ is a solution of equation (\ref{equa}) such that
 \begin{eqnarray}\label{17}
  \forall s \ge s^*, E(w(s),\partial_s w(s)) \ge E(\kappa_0,0),
 \end{eqnarray}
and
 \begin{eqnarray}\label{18}
\Big|\Big|\begin{pmatrix} w(s^*)\\\partial_s w(s^*) \end{pmatrix} -e^{i \theta^*}\begin{pmatrix} \kappa(\bb ^*,\cdot)\\0\end{pmatrix} \Big|\Big|_{\mathcal{ H}}\le \epsilon^*
 \end{eqnarray}
for some $\bb^* \in (-1,1), \theta^* \in \mathbb{R}$ and $\epsilon^*\in (0,\epsilon_0]$, 
then there exists $\bb_{\infty} \in (-1,1)$ and $\theta_\infty\in \mathbb{R}$ such that
$$|\atang \,\bb_{\infty}-\atang \,                                                                                                                                                                                                                                                                                                                                                  \bb^*|+|\theta_\infty-\theta^*|\le C_0 \epsilon^* $$
and for all $s\ge s^*$:
 \begin{eqnarray}\label{19}
\Big|\Big|\begin{pmatrix} w(s)\\\partial_s w(s) \end{pmatrix} -e^{i \theta_{\infty}}\begin{pmatrix} \kappa(\bb _\infty,\cdot)\\0\end{pmatrix} \Big|\Big|_{\mathcal{ H}}\le C_0\epsilon^* e^{-\mu_0(s-s^*)}.
 \end{eqnarray}
\end{theorr}
\noindent{\bf Remark:} Condition (\ref{17}) is crucial for the conclusion. Indeed, if (\ref{18}) is satisfied but not (\ref{17}), we may have a different conclusion, as with the explicit solution $w(y,s)=\kappa_0 \frac{(1-\bb^2)^\frac{1}{p-1}}{(1+\mu e^s+\bb y)^\frac{2}{p-1}}$ which may be made arbitrarily close to $\kappa(\bb ,y)$ and satisfies convergences to $0$ as $s\rightarrow +\infty$.\\
As we said earlier, the third step towards the derivation of the blow-up profile simply uses Proposition \ref{2} and Theorem \ref{theo3} to get get the following:
\begin{theorr}\label{theo4}{\bf (Blow-up profile near a non-characteristic point)}
 If $u$ is a solution of (\ref{waveq}) with blow-up curve $\Gamma :\{ x \rightarrow T(x)\}$ and $x_0 \in \mathbb R$ is non-characteristic (in the sense (\ref{4})), then there exist $\bb_\infty (x_0) \in (-1,1),\, \theta_\infty (x_0) \in \mathbb{R}$ and $s^*(x_0) \ge -\log T(x_0)$ such that for all $s\ge s^*(x_0),$ (\ref{19}) holds with $\epsilon^*=\epsilon_0$, where $C_0$ and $\epsilon_0$ are given in Theorem \ref{theo3}. Moreover,
 \begin{equation*}
||w_{x_0}(s)-e^{i\theta_\infty(x_0)} \kappa(\bb _\infty (x_0))||_{H^1(-1,1)}+||\partial_s w_{x_0}(s)||_{L^2(-1,1)} \rightarrow  0 \mbox{ as } s  \rightarrow \infty.
  \end{equation*}
\end{theorr}
\noindent{\bf Remark:} From the Sobolev embedding, we know that the convergence takes place also in $L^\infty$, in the sense that
$$||w_{x_0}(s)-e^{i\theta_\infty(x_0)} \kappa(\bb _\infty (x_0))||_{L^\infty(-1,1)} \rightarrow  0 \mbox{ as } s  \rightarrow \infty.$$
\noindent{\bf Remark:} As we mentioned above one of our difficulties comes from the invariance of the solution under complex rotation, which induces an additional zero-mode in the linearization of equation (\ref{equa}) around $\kappa(\bb ,y)$. In order to overcome that difficulty, we use a modulation technique (see Proposition \ref{5.1} below). Let us mention that the extension from the real to complex case has been successfully performed by Filippas and Merle \cite{MR1317705} in the case of the semilinear heat equation with Sobolev subcritical nonlinearity:
\begin{equation} \label{chaleur}
\partial_{t} u = \Delta u+|u|^{p-1}u, \mbox{ with }p>1\mbox{ and }(N-2)p<N+2.
\end{equation} 
We mention that that the modulation technique was already crucial in \cite{MR1317705} to control the additional zero-mode coming from the invariance of equation (\ref{chaleur}) under complex rotation. Note however that the adaptation from the real to the complex case for the wave equation is far more difficult, since we have additional problems, coming from the fact that we have to handle non self-adjoint operators.

\bigskip
This paper is organized as follows:\\
 In Section 2, we
 characterize the set of stationary solutions, proving Proposition \ref{p1}.\\
In Section 3, we study the properties of the linearized operator of equation (\ref{equa}) around a non-zero stationary solution.\\
Finally, Section 4 is devoted to the proof of Proposition \ref{2}, Theorem \ref{theo3} and Theorem \ref{theo4}.\\

\section{Characterization of the stationary solutions in self-similar variables}
In this section, we prove Proposition \ref{p1} which characterizes all $\mathcal{H}_0$ solutions of
\begin{equation}
 \frac{1}{\rho} (\rho (1-y^2) w')'-\frac{2(p+1)}{(p-1)^2}w+|w|^{p-1}w=0,\label{47}
\end{equation}
the stationary version of (\ref{equa}). As our solution is a complex-valued one, we will use in addition to the techniques of Section 2.3 in \cite{MR2362418}, other techniques as the determination of the phase and some projections. Note that since $0$ and $\kappa_0 e^{i\theta}$ are trivial solutions to equation (\ref{equa}) for any $\theta \in \mathbb R$, we see from Lemma \ref{2.6}
 that $\mathcal{T}_\bb  (e^{i\theta}\kappa_0)=e^{i\theta} \kappa(\bb ,y)$ is also a stationary solution to (\ref{equa}). Let us introduce the set
\begin{equation}\label{48}
 S\equiv \{0, e^{i\theta} \kappa (\bb ,\cdot), |\bb|<1, \theta \in \mathbb{R}\}
\end{equation}
and prove that there are no more solutions of (\ref{47}) in $\mathcal{H}_0$ outside the set $S$.
\begin{proof} We first prove $(ii)$, since its proof is short.\\
(ii) Since we clearly have from the definition (\ref{15}) that $E(0,0)=0$, we will compute $E(e^{i\theta}\kappa(\bb ,\cdot),0)$. 
From (\ref{15}) and the proof of the real case page 59 in \cite{MR2362418}, we see that
$$ E(e^{i\theta}\kappa(\bb ,\cdot),0)= E(\kappa(\bb ,\cdot),0)=E(\kappa_0,0)>0.$$
Thus, (\ref{14}) follows.\\
(i) Consider $w \in \mathcal{H}_0$ a non-zero solution of (\ref{47}). Let us prove that there are some $\bb \in (-1,1)$ and $\theta \in \mathbb{R}$ such that $w=e^{i\theta}\kappa(\bb ,\cdot).$ For this purpose, define
\begin{equation}\label{50}
 \xi=\frac{1}{2} \log\left(\frac{1+y}{1-y}\right) (\mbox{that is }  y=\tanh \xi)\mbox{ and } v( \xi)=w(y) (1-y^2)^\frac{1}{p-1}. 
\end{equation}
As in the real case, we see from straightforward calculations that $v\not\equiv 0$ is a $H^1(\mathbb{R})$ solution to
\begin{equation}\label{51}
 \partial_\xi^2 v +|v|^{p-1}v-\frac{4}{(p-1)^2}v=0,\,\forall \xi \in \mathbb{R}.
\end{equation}
\noindent Our aim is to prove the existence of $\theta_0 \in \mathbb{R}$ and $\xi_0 \in \mathbb{R}$ such that $v(\xi)=e^{i\theta_0}\check {k} (\xi+\xi_0)$ where 
$$\check {k}(\xi)=\frac{\kappa_0}{\cosh^\frac{2}{p-1}(\xi)}.$$
% is the solution of the KdV equation, with $\check k'(0)=0$.\\
Since $v \in H^1(\mathbb R)  \subset C^\frac{1}{2}(\mathbb R),$ we see that $v$ is a strong $C^2$ solution of equation (\ref{51}). Since $v\not\equiv  0$, there exists $\xi_0 \in \mathbb{R}$ such that $v(\xi_0)\neq 0$. By invariance of (\ref{51}) under translation, we may suppose that $\xi_0=0$.
Let
 \begin{equation}\label{*}
  \g^*=\left\{ \xi \in \mathbb{R} \,|\, \check w(\xi)\neq 0 \right\},
 \end{equation}
 a nonempty open set by continuity. Note that $\g^*$ contains some non empty interval $I$ containing 0. We also introduce $\rr(\xi)=|v(\xi)|$, $\theta$ and $ \check \theta$ two determinations of the phase given by
 \begin{equation}\label{la}
\theta(\xi)=\arctan \left(\frac{ \operatorname{Im} v(\xi)}{ \operatorname{Re} v(\xi)} \right) \mbox{ and } \check\theta(\xi)=\mbox{arccotan}\left(\frac{ \operatorname{Re} v(\xi)}{ \operatorname{Im} v(\xi)}\right),
 \end{equation}
and $h:\g^* \to \mathbb{R}$ given by 
\begin{equation}
\forall \; \xi \in \g^*, \,h(\xi)=
\left\{
\begin{array}{l}
\theta'(\xi),\, \mbox{if }  \operatorname{Re} v(\xi)\neq 0,\\
\check \theta'(\xi),\, \mbox{if }  \operatorname{Im} v(\xi)\neq 0.
\end{array}
\right . \label{h}
\end{equation}
\noindent We claim that $h$ is well defined and that $h\in C^1(\g^*)$. Indeed, let $\xi_0 \in \mathbb{R}$ such that $v(\xi_0)\neq 0$. Necessarily, either its real or its imaginary part is nonzero. If for instance $\operatorname{Re} v(\xi_0 )\neq 0$, by continuity
$$\exists \delta_0>0,\, \forall  |\xi-\xi_0|<\delta_0, \,\operatorname{Re} \check w(\xi)\neq 0,$$
so $\theta$ is well defined in $(\xi_0 - \delta_0,\xi_0+\delta_0), $ and  $h$ is well defined and $C^1$ in $(\xi_0 - \delta_0,\xi_0+\delta_0)$. Now, if $\operatorname{Im} v(\xi_0 )\neq 0 $, by the same way, we prove that $h$ given by $h(\xi)=\check \theta'(\xi)$ is well defined and $C^1$ in a small interval $(\xi_0 - \delta_0,\xi_0+\delta_0) $.\\
\noindent This definition is nonambiguous. Indeed, if ever both $\theta$ and $\check \theta$ are defined on the same interval $(a,b)$ with $a<b$, then there exists $k \in \mathbb {Z}$ such that
$$\forall \xi \in (a,b),\, \theta (\xi)=\check \theta (\xi)+2k\pi.$$
Differentiating this, we get
$$\forall \xi \in (a,b),\, \theta' (\xi)=\check \theta' (\xi).$$
Thus, (\ref{h}) defines $h(\xi)$ with no ambiguity.\\
Take $\xi \in \g^*$. Using one of the angle determination in (\ref{la}) and
projecting equation (\ref{51}), 
 we see that
 \begin{equation}\label{mer}
 \forall \xi \in \g^*,\,\left\{
\begin{array}{l}
 \rr''(\xi)-\rr (\xi)h^2(\xi)-c_0\rr(\xi)+\rr^p(\xi)=0,\; c_0= \frac{4}{(p-1)^2}\\
2\rr'(\xi)h(\xi)+\rr(\xi) h'(\xi)=0.
\end{array}
 \right . 
\end{equation}
Integrating the second equation on the interval $I\subset \g^*$, we see that for all $\xi \in I$, $h(\xi)=\frac{h(0)\rr^2(0)}{\rr^2(\xi)}$. Plugging this in the first equation, we get
 \begin{equation}\label{stka}
 \forall \xi \in I,\,\rr''(\xi)-\frac{\mu}{\rr^3(\xi)}-c_0\rr(\xi)+\rr^p(\xi)=0 \mbox{ where }\mu=h^2(0)\rr^4(0).
 \end{equation}
\noindent Now let
 \begin{eqnarray}\label{waw}
 \g= \left\{ \xi \in \g^*, \forall \xi'\in I_\xi,\, h(\xi')=\frac{h(0)\rr^2(0)}{\rr^2(\xi')} \right\},
 \end{eqnarray}
where $I_\xi=[0,\xi)$ if $\xi\ge 0$ or $I_\xi=(\xi,0]$ if $\xi\le 0$. Note that $I \subset G$. Now, we give the following:
\begin{lem} \label{tunisie}
 There exists 
$\epsilon_0 >0$ such that 
$$\forall \xi \in G, \,\forall \xi'\in  I_\xi,\, 0< \epsilon_0 \le |v(\xi')|\le \frac{1}{\epsilon_0}.$$

\end{lem}
\begin{proof}
Take $\xi \in \g$. By definition (\ref{waw}) of $\g$, we see that equation (\ref{stka}) is satisfied for all $\xi'\in I_\xi$. Multiplying $\rr''(\xi)-\frac{\mu}{\rr^3(\xi)}-c_0\rr(\xi)+\rr^p(\xi)=0  $ by $\rr'$ and integrating between $0$ and $\xi$, we get:
 $$\forall \xi  \in  I_\xi,\, \mathcal{E}(\xi')= \mathcal{E}(0), \mbox{ where } \mathcal{E}(\xi')=\frac{1}{2}\rr'^2(\xi')+\frac{\mu}{2 \rr^2(\xi')}-\frac{c_0}{2}\rr^2(\xi')+\frac{\rr^{p+1}(\xi')}{p+1},$$
\noindent or equivalently, 
$$\forall \xi' \in I_\xi,\, F(\rr(\xi'))=\frac{1}{2} \rr'(\xi')^2\ge 0 \mbox{ where }F(r)=-\frac{\mu}{2r^2}+\frac{c_0}{2} r^2-\frac{r^{p+1}}{p+1}+ \mathcal{E}(0).$$
Since $F(r)\rightarrow -\infty$ as $r \rightarrow 0$ or $r \rightarrow \infty$, there exists $\epsilon_0=\epsilon_0(\mu, E(0)) >0$ such that $\epsilon_0\le \rr(\xi') \le \frac{1}{\epsilon_0}$, which yields to the conclusion of the Claim \ref{tunisie}.
\end{proof}
We claim the following:
\begin{lem}\label{os} It holds that
$\g=\mathbb{R}$.
\end{lem}
\begin{proof}
Note first that by construction, $\g$ is a nonempty interval (note that $0\in I \subset\g$ where $I$ is defined right after (\ref{*})). We have only to prove that $\sup \g=+\infty$, since the fact that $\inf \g=-\infty$ can be deduced by replacing $v(\xi)$ by $v(-\xi)$.\\By contradiction, suppose that $\sup \g=a<+\infty$. 
By continuity, we have
\begin{equation}\label{gg}
\forall \xi \in [0, a], h(\xi)=\frac{h(0)\rr(0)^2}{\rr(\xi)^2},
\end{equation}
on the one hand. On the other hand, by Lemma \ref{tunisie}, for all $ \xi' \in [0,a),  0< \epsilon_0 \le |v(\xi')|\le \frac{1}{\epsilon_0}$. Using (\ref{gg}) $v(a)\neq 0$, and $a\in \g^*$.
Using (\ref{gg}), we see that $a\in \g$.
By continuity, we can write for all $\xi \in (a-\delta,a+\delta),$ where $\delta >0 $ is small enough,
 \begin{equation*} \left\{
\begin{array}{l}
 \rr''(\xi)-\rr (\xi)\theta'(\xi)^2-c_0\rr(\xi)+\rr(\xi)^p=0,\; c_0= \frac{4}{(p-1)^2}\\
2\rr'(\xi)h(\xi)+\rr(\xi) h'(\xi)=0.
\end{array}
 \right . 
\end{equation*}
From the second equation and (\ref{gg}) applied with $\xi=a$, we see that $h(\xi)=\frac{h(a)(\rr(a))^2}{(\rr(\xi))^2}=\frac{h(0)(\rr(0))^2}{(\rr(\xi))^2}$. Therefore, it follows that $(a, a+\delta )\in E$, which contradicts the fact that $a=\sup\g$.
\end{proof}
Note from Lemma \ref{os} that (\ref{mer}) and (\ref{stka}) hold for all $\xi \in \mathbb{R}$.
 We claim that $h(0)=0$. Indeed, if not, then by (\ref{stka}), we have $\mu \neq 0$, and since $\g=\mathbb{R}$, we see from Lemma \ref{tunisie} that for all $\xi\in\mathbb{R}$, $|v(\xi)|\ge \epsilon_0$, therefore $v \notin L^2(\mathbb{R})$, which contradicts the fact that $v\in H^1(\mathbb{R})$. Thus, $h(0)=0$, and $\mu=0$.
 By uniqueness of solutions to the second equation of (\ref{mer}), we see that $h(\xi)=0$ for all $\xi \in \mathbb{R}$, so $\theta(\xi)=\theta(0)$ and $v'(0)=\lambda e^{i\theta (0)}\, (\lambda \in \mathbb{R})$. Thus
 \begin{equation*} \left\{
\begin{array}{l}
v(0)=R(0)e^{i\theta (0)}\\
v'(0)=\lambda e^{i\theta (0)}.
\end{array}
 \right . 
\end{equation*}
Let $W\in H^1(\mathbb{R})$ be a real solution of
 \begin{equation}\label{po} \left\{
\begin{array}{l}
W''-c_0 W+|W|^{p-1}W=0\\
 W(0)=R(0)\\
 W'(0)=\lambda.
\end{array}
 \right . 
\end{equation}
By uniqueness of the Cauchy problem of equation (\ref{51}), we have for all $\xi \in \mathbb{R}, v(\xi)=W(\xi) e^{i\theta(0)}$, and as $v\in H^1(\mathbb{R})$, $W$ is also in $H^1(\mathbb{R})$. Since it is well known that the real solutions of (\ref{po}) in $H^1(\mathbb{R})$ are
\begin{equation}\label{essou}
 \mbox{either } W \equiv 0 \;\mbox{or }W(\xi)=\pm\check {k }(\xi+\xi_0)\,\mbox{ for some }\xi_0 \in \mathbb{R},
\end{equation}
for the reader's convenience, we recall the proof in Appendix \ref{appendix1}, it follows that $v(\xi)=e^{i\theta_0}\check {k }(\xi+\xi_0)$ for some $\xi_0 \in \mathbb{R}$, because $W \not\equiv 0$ and $W>0$. Thus, for $d=\tanh \xi_0\,\in\, (-1,1)$ and $y=\tanh \xi$, we get
\begin{align*}
 &v(\xi)=e^{i\theta_0}\kappa_0 \left[1-\tanh (\xi+\xi_0)^2\right]^\frac{1}{p-1}=e^{i\theta_0}\kappa_0 \left[1-\left(\frac{\tanh \xi+\tanh \xi_0}{1+\tanh \xi \tanh \xi_0}\right)^2\right]^\frac{1}{p-1}\notag\\&=e^{i\theta_0}\kappa_0 \left[1-\left(\frac{y+\bb }{1+\bb y}\right)^2\right]^\frac{1}{p-1}=e^{i\theta_0}\kappa_0 \left[\frac{(1-\bb^2)(1-y^2)}{(1+\bb y)}^2\right]^\frac{1}{p-1}=e^{i\theta_0}\kappa(\bb ,y) (1-y^2)^\frac{1}{p-1}.
\end{align*}
By (\ref{50}), we see that $w(y)=e^{i\theta_0}\kappa(\bb ,y)$. This concludes the proof of Proposition \ref{p1}.
% By simple calculation we see that there exists $\xi_0$ such that $W(\xi_0)=C$ and $W'(\xi_0)=0$, as we have for all $\xi \in \mathbb{R}$, $\check {k }(\xi+\xi_0)$ satisfies (\ref{po}) and by uniqueness of the Cauchy problem of equation (\ref{po}), we have 
% $$\forall \xi \in \mathbb{R},\, W(\xi)=\check {k }(\xi+\xi_0).$$ 
% In addition, for $\theta_0=\theta(0)$, $v(\xi)=e^{i\theta_0}\check {k }(\xi+\xi_0)$.
\end{proof}
\section{The linearized operator around a non zero stationary solution}
In this section, we study the properties of the linearized operator of equation (\ref{equa}) around the stationary solution $ \kappa (\bb ,y)$ (\ref{defk}). We recall that in \cite{MR2362418}, the authors have treated the real case, by introducing $q=(q_1,q_2)\in \mathbb{R}\times \mathbb{R}$ and linearizing around $ \kappa (\bb ,y)$. It turns out that the real part of our linearized operator for complex-valued solution is identical to the real part of the linearized operator in the real case treated in \cite{MR2362418}. As a matter of fact, we will rely on \cite{MR2362418} for the real part and have to invent new methods for the imaginary part.

\noindent For any complex number $z$, we use in the following the notation
\begin{align*}
 \Check z=\operatorname{Re}(z)\mbox{ and }\tilde z=\operatorname{Im}(z).
\end{align*}
If we introduce $q=(q_1,q_2)=\begin{pmatrix}{q_1}\\{q_2}\end{pmatrix}\in \mathbb{C}\times \mathbb{C}$
for all $ s\in[s_0, \infty)$, for a given $s_0 \in \mathbb{R}$, by
\begin{eqnarray*}
\begin{pmatrix} w(y,s)\\\partial_s w(y,s) \end{pmatrix} =\begin{pmatrix} \kappa(\bb ,y)\\0\end{pmatrix} +\begin{pmatrix} q_1(y,s)\\q_2(y,s)\end{pmatrix},
\end{eqnarray*}
then, we see from equation (\ref{equa}) that $q$ satisfies the following equation for all $s\ge s_0$:
\begin{eqnarray} \label{linearis�-complexe}
\frac{\partial }{\partial s} \begin{pmatrix} q_1\\q_2 \end{pmatrix} = L_\bb \begin{pmatrix}q_1\\q_2\end{pmatrix}+\begin{pmatrix}0\\ f_\bb (q_1)\end{pmatrix},
\end{eqnarray}
where
$$ L_\bb \begin{pmatrix}q_1\\q_2\end{pmatrix}=\begin{pmatrix}q_2 \\\mathcal{L}q_1+\check\psi (\bb , y)\check q_1+i \tilde\psi (\bb , y)\tilde{q_1}-\frac{p+3}{p-1} q_2- 2 y \partial_y q_2 \end{pmatrix}, $$
$$\check\psi (\bb , y)=p \kappa(\bb , y)^{p-1}-\frac{2(p+1)}{(p-1)^2},$$
$$\tilde\psi (\bb , y)= \kappa(\bb , y)^{p-1}-\frac{2(p+1)}{(p-1)^2},$$
$$f_\bb (q_1)=\check{f_\bb }(\check q_1,\tilde q_1)+i\tilde{f_\bb }(\check q_1,\tilde q_1),$$
$$ \mbox{where } \check{f_\bb }(\check q_1,\tilde q_1)=|\kappa(\bb , y)+q_1|^{p-1}(\kappa(\bb , y)+\check q_1)-\kappa(\bb , y)^{p}-p\kappa^{p-1}(\bb , y)\check q_1,$$
\begin{eqnarray}\label{barL_bb }\tilde{f_\bb }(\check q_1,\tilde q_1)=|\kappa(\bb , y)+q_1|^{p-1}\tilde q_1-\kappa^{p-1}(\bb , y)\tilde q_1.\end{eqnarray}
From (\ref{linearis�-complexe}), dissociating the real and imaginary parts, we get for all $s\ge s_0$:
\begin{eqnarray}\frac{\partial }{\partial s}\begin{pmatrix} \check q_1\\\check q_2
\end{pmatrix}=\check L_\bb  \begin{pmatrix} \check q_1 \\\check q_2 \end{pmatrix}+\begin{pmatrix}0\\ \check{f_\bb }\end{pmatrix},
\end{eqnarray}
where
\begin{eqnarray}\label{barL_bb }
\check L_\bb  \begin{pmatrix} \check q_1\\\check q_2 \end{pmatrix}=
\begin{pmatrix}\check q_2 \\\mathcal{L} \check q_1+\check\psi (\bb , y)\check q_1-\frac{p+3}{p-1} \check q_2- 2 y \partial_y \check q_2
\end{pmatrix},
\end{eqnarray}
and
\begin{eqnarray}\frac{\partial }{\partial s} \begin{pmatrix} \tilde{q_1}\\\tilde q_2
\end{pmatrix}=\tilde L_\bb \begin{pmatrix} \tilde q_1\\\tilde q_2 \end{pmatrix}+\begin{pmatrix} 0\\\tilde{f_\bb } \end{pmatrix}, 
\end{eqnarray}
where
\begin{eqnarray}\label{tildeL_bb }
 \tilde L_\bb \begin{pmatrix} \tilde q_1\\\tilde q_2 \end{pmatrix}=\begin{pmatrix} \tilde q_2\\\mathcal{L}\tilde q_1+\tilde\psi (\bb , y)\tilde q_1-\frac{p+3}{p-1} \tilde q_2- 2 y \partial_y \tilde q_2 \end{pmatrix}.
\end{eqnarray}
\noindent{\bf Remark:} From (\ref{linearis�-complexe}) we see that for $q=\begin{pmatrix}\check {q_1}\\\check {q_2}\end{pmatrix}+i\begin{pmatrix}\tilde{q_1}\\\tilde{q_2}\end{pmatrix}, $ we have
 \begin{equation*}\label{+}
  L_\bb  \begin{pmatrix}q_1\\q_2\end{pmatrix}= \check L_\bb  \begin{pmatrix}\check {q_1}\\ \check {q_2}\end{pmatrix}+i \tilde L_\bb  \begin{pmatrix}\tilde{q_1}\\ \tilde{q_2}\end{pmatrix}.
 \end{equation*}
Note that the operator $\check L_\bb $ (\ref{barL_bb }) already appears in the real case studied in \cite{MR2362418}. For that reason, we recall from that paper the properties of $\check L_\bb $, and focus here on the properties of $\tilde L_\bb $, which is one of the novelties of our work.
\noindent Note from (\ref{9}) that we have
$$||q||_\mathcal{H}=[\phi (q, q)]^\frac{1}{2}<+\infty,$$
where the hermitian inner product $\phi$ is defined by
\begin{equation*}
 \phi (q, r)=\phi \left(\begin{pmatrix} q_1\\q_2 \end{pmatrix}, \begin{pmatrix} r_1\\r_2 \end{pmatrix}\right)=\int_{-1}^{1}(q_1 \bar r_1+q'_1 \bar r'_1(1-y^2)+q_2 \bar r_2)\rho(y)\;dy.
\end{equation*}
Using integration by parts and the definition of $\mathcal{L}$ (\ref{8}), we have the following:
\begin{equation}\label{phi}
 \phi (q, r)=\int_{-1}^{1}(q_1(-\mathcal{L} \bar r_1+ \bar r_1)+q_2  \bar r_2)\rho(y) \,dy.
\end{equation}
We note that $q\in \mathcal{ H}$ if and only if $\check q\in \mathcal{ H}$ and $\tilde q\in \mathcal{ H}$, and
$$||q||_{\mathcal{ H}}^2=||\check q||_{\mathcal{ H}}^2+||\tilde q||_{\mathcal{ H}}^2.$$
This section is organized as follows:

\medskip
-We first recall some spectral properties of $\check L_\bb $ which was proved by Merle and Zaag in \cite{MR2362418}.

\medskip
-Then, we focus on the study of $\tilde L_\bb $, precisely, we compute $\tilde L_\bb ^*$ the conjugate operator of $\tilde L_\bb $ and we give a zero direction for it.

\medskip
-Using the projection on the eigenspace of $\tilde L_\bb $, we introduce a function which will capture the dispersive character of equation (\ref{linearis�-complexe}), and give some dispersive estimates in order to prove Theorem \ref
{theo3}.
\subsection{Spectral properties of $\check L_\bb $}
From Section 4 in \cite{MR2362418}, we know that $\check L_\bb $ has two nonnegative eigenvalues $\lambda=1$ and $\lambda=0$ with eigenfunctions
\begin{equation}
\check F_1 (\bb ,y)= (1-\bb^2)^{\frac{p}{p-1}}\begin{pmatrix}(1+\bb y)^{-\frac{p+1}{p-1}}\\(1+\bb y)^{-\frac{p+1}{p-1}}\end{pmatrix}
\mbox{and }\;
\check F_0(\bb ,y)= (1-\bb^2)^{\frac{1}{p-1}}\begin{pmatrix}\frac{y+\bb }{(1+\bb y)^\frac{p+1}{p-1}}\\0\end{pmatrix}. \label{110}
\end{equation}
Note that for some $C_0>0$ and any $\lambda\in \{0,1\}$, we have
\begin{equation}\label{majoration}
\forall |\bb|<1,\;\; \frac{1}{C_0}\leq ||\check F_\lambda^\bb ||_{\mathcal{H}} \leq C_0 \;\mbox{  and  }\; ||\partial_\bb  \check F_\lambda^\bb ||_{\mathcal{H}} \leq \frac{C_0}{1-\bb^2}.
\end{equation}
We know also that $\check L_\bb ^*$ the conjugate operator of $\check L_\bb $ with respect to $\phi$ is given by
\begin{equation*}
\check L_\bb ^*
\begin{pmatrix}
r_1\\
r_2\end{pmatrix}= \begin{pmatrix}
\check R_\bb (r_2)\\
-\mathcal{L} r_1+r_1+\frac{p+3}{p-1}r_2+2y r_2'-\frac{8}{(p-1)}\frac{r_2}{(1-y^2)}\end{pmatrix}
\end{equation*}
for any $(r_1, r_2)\in (\mathcal{D}(\mathcal{L}))^2$, where $r=\check R_\bb  (r_2)$ is the unique solution of
\begin{equation*}
-\mathcal{L} r+r=\mathcal{L} r_2+\check{\psi} (\bb , y)r_2. 
\end{equation*}
Here, the domain $\mathcal{D}(\mathcal{L})$ of $\mathcal{L}$ defined in (\ref{8}) is the set of all $r \in L_{\rho}^2$ such that $\mathcal{L} r \in L_{\rho}^2.$
\medskip

\noindent Furthermore, $\check L_\bb ^*$ has two nonnegative eigenvalues $\lambda=0$ and $\lambda=1$ with eigenfunctions $\check W_{\lambda}$ such that 
\begin{equation*}
\check W_{1, 2}(\bb  ,y)= \check c_1 \frac{(1-y^2)(1-\bb )^\frac{1}{p-1}}{(1+\bb y)^\frac{p+1}{p-1}},\,\check W_{0, 2}(\bb  ,y)= \check c_0 \frac{(y+\bb )(1-\bb )^\frac{1}{p-1}}{(1+\bb y)^\frac{p+1}{p-1}},
\end{equation*}
with\footnote{ In section 4 of \cite{MR2362418}, we had non explicit normalizing constants $\check c_\lambda=\check c_\lambda (\bb )$. In Lemma 2.4 in \cite{MZ13}, the authors compute the explicit dependence of $\check c_\lambda (\bb )$.}
\begin{equation*}
\frac{1}{\check c_\lambda}=2(\frac{2}{p-1}+\lambda)\int_{-1}^1 (\frac{y^2}{1-y^2} )^{1-\lambda}\rho(y) \,dy,
\end{equation*}
and $\check W_{\lambda, 1}(\bb,\cdot)$ is the unique solution of the equation 
\begin{equation*}
-\mathcal{L} r+ r =\left(\lambda-\frac{p+3}{p-1}\right) r_2- 2 y r'_2 + \frac{8}{p-1} \frac{r_2}{1-y^2}
\end{equation*}
with $r_2= \check W_{\lambda, 2}(\bb,\cdot)$. Note that we have the following relations for $\lambda=0$ or $\lambda=1$
\begin{equation*}
\phi (\check W_\lambda(\bb,\cdot) ,\check F_\lambda(\bb,\cdot) )=1\mbox{ and }\phi (\check W_\lambda (\bb,\cdot),\check F_{1-\lambda}(\bb,\cdot) )=0.
\end{equation*}
Let us introduce for $\lambda \in \{0, 1\}$ the projectors $\check \pi^\bb _\lambda(r)$, and $\check \pi_-^\bb (r)$ for any $r\in\mathcal{H}$ by
\begin{equation} \label{barpi}
\check \pi_\lambda^\bb  (r)=\phi (\check W_\lambda (\bb,\cdot) , r), 
\end{equation}

\begin{equation} \label{ma}
r=\check \pi_0^\bb  (r) \check F_0(\bb  ,\cdot)+\check \pi_1^\bb  (r) \check F_1(\bb  ,\cdot)+ \check \pi_{-}^\bb  (r),
\end{equation}
and the space
\begin{equation*}
\check{\mathcal{H}}_{-}^\bb  \equiv \{r \in \mathcal{H} \,|\, \check \pi_1^\bb  (r)=\check \pi_0^\bb (r)=0\}.
\end{equation*}
Introducing the bilinear form
\begin{eqnarray}\label{134'}
 \check \varphi_{\bb } (q,r)&=&\int_{-1}^{1} (-\check \psi(\bb ,\cdot) q_1r_1+q_1' r_1'(1-y^2)+q_2 r_2 ) \rho dy,
\end{eqnarray}
where $\check \psi(\bb ,y)$ is defined in (\ref{barL_bb }), we recall from Proposition 4.7 page 90 in  \cite{MR2362418} that there exists $C_0>0$ such that for all $|\bb|<1$, for all $r\in \tilde{\mathcal{H}}^\bb _{-},$
\begin{align}\frac{1}{C_0} ||r||_{\mathcal{H} }^2\le \check \varphi_\bb  (r,r)\le C_0 ||r||_\mathcal{H} ^2.
\label{36,5}
\end{align}

\noindent In the following sections, we follow the method of \cite{MR2362418} to study the spectral properties of $\tilde L_\bb $.

\subsection{A zero direction of $\tilde L_\bb $}
Let us show that $\lambda=0$ is an eigenvalue for $\tilde L_\bb $. We claim the following:
\begin{lem}{\bf (Zero direction of $\tilde L_\bb $)}\label{fonctions_propres}
\item{(i)} For all $|\bb|< 1$, $\lambda= 0$ is an eigenvalue of the linear operator $\tilde L_\bb $ and its corresponding eigenfunction is 
\begin{equation}\label{F}
\tilde F_0(\bb  ,y)= \begin{pmatrix}\kappa (\bb ,y)\\0\end{pmatrix}.
\end{equation}
\item{(ii)} Moreover, it holds for some $C_0>0$ that
\begin{equation}\label{majoration1}
\forall |\bb|<1,\;\; \frac{1}{C_0}\leq ||\tilde F_0(\bb,\cdot)||_{\mathcal{H}} \leq C_0 \;\mbox{  and  }\; ||\partial_\bb  \tilde F_0(\bb,\cdot)||_{\mathcal{H}} \leq \frac{C_0}{1-\bb^2}.
\end{equation}
\end{lem}
\noindent {\bf Remark}: There is a more geometrical way to see that $\lambda=0$ is an eigenvalue for $\tilde L_\bb$ and $\check L_\bb$ (in other worlds, a double eigenvalue for $L_\bb$ given in (\ref{barL_bb })): simply note that equation (\ref{equa}) has a 2-parameter family of stationary solutions
$$K(\bb,\theta,y)=\left(e^{i\theta}\kappa(\bb,y),0\right)$$ 
hence, $\partial_\bb K(\bb,0,y)=\left(\partial_\bb \kappa(\bb,y),0\right)$ and $\partial_\theta K(\bb,0,y)=\left(i \kappa(\bb,y),0\right)$ are eigenfunctions of the linearized operator of equation (\ref{equa}) around $K(\bb,0,y)=\left( \kappa(\bb,y),0\right)$, which is precisely the operator $L_\bb$. Splitting $L_\bb$ into real and imaginary parts shows that $(\partial_\bb \kappa(\bb,y),0)$ and $ (\kappa(\bb,y),0)$ are eigenfunctions of $\check L_\bb$ and $\tilde L_\bb$, respectively. A simple calculation shows indeed that $(\partial_\bb \kappa(\bb,y),0)$ is proportional to $\tilde F_0(\bb  ,y)$ given in.(\ref{110}).

\noindent The fact that $\lambda=1$ is an eigenvalue of $L_\bb$ follows from similar ideas: noting that
\begin{eqnarray*}
 \bar K(\bb,\mu,y,s)=\left( \frac{\kappa_0(1-\bb^2)^\frac{1}{p-1}}{(1+\mu e^s+\bb y)^\frac{2}{p-1}},-\frac{2\mu e^s}{p-1}\frac{\kappa_0 (1-\bb^2)^\frac{1}{p-1}}{(1+\mu e^s+\bb y)^\frac{p+1}{p-1}}\right)
\end{eqnarray*}
is an explicit solution of equation (\ref{equa}) with $ \bar K(\bb,0,y,s)=(\kappa(\bb,y),0)$, when $\mu=0$, differentiating with respect to the new parameter $\mu$, we obtain an eigenfunction for $L_\bb$ with $\lambda=1$.

\begin{proof}
 $(i)$ As $\kappa (\bb ,y)$ is a stationary solution of (\ref{equa}), it satisfies (\ref{47}), hence
$$\mathcal{L}\kappa(\bb ,y)- \frac{2(p+1)}{(p-1)^2}\kappa(\bb ,y)+{\kappa(\bb ,y)}^p=0. $$
By definition (\ref{tildeL_bb }) of $\tilde L_\bb $, we see that $\tilde L_\bb \begin{pmatrix}\kappa (\bb ,y)\\0\end{pmatrix}=0.$\\
$(ii)$ Noting that $\kappa(\bb ,y)=\mathcal{T}_\bb  (\kappa_0) $ where the transformation $\mathcal{T}_\bb  $ is defined in $(\ref{transformation})$, applying Lemma \ref{ff} and using $(\ref{defk})$, we get the first bound. In order to prove the second one, we recall the following integral calculation rules from \cite{MR2362418}:
\begin{cl} \label{claim}
Consider for some $\alpha >-1$ and $\beta \in \mathbb{R}$ the following integral: 
$$I(\bb )=\int_{-1}^{1} \frac{(1-y^2)^{ \alpha}}{(1+\bb y)^{\beta}}\; dy.$$
\item{(i)} if $\alpha+1-\beta >0$, then $\frac{1}{K}\leq I(\bb ) \leq K$;
\item{(ii)} if $\alpha+1-\beta =0$, then $\frac{1}{K}\leq\frac{ I(\bb )}{|\log(1-\bb^2)|} \leq K$;
\item{(iii)} if $\alpha+1-\beta< 0$, then $\frac{1}{K}\leq I(\bb )(1-\bb^2)^{-(\alpha+1)+\beta} \leq K$.
\end{cl}
\begin{proof}
See page 84 of \cite{MR2362418}.
\end{proof}
Using the definition of $\tilde{F_0}$ (\ref{F}), the fact that
\begin{equation}\label{naro}
 \forall (\bb ,y)\in (-1,1)^2, |y+\bb |+|1-\bb^2|+(1-y^2) \leq  C (1+\bb y),
\end{equation}
and straightforward computations we see that
\begin{center}
\begin{tabular}{ll} 
 $|\partial_\bb  \tilde F_0(\bb  ,y)| \leq C\frac{(1-\bb^2)^{\frac{2-p}{p-1}}}{(1+\bb y)^\frac{2}{p-1}}, $&$ |\partial^2_{\bb ,y} \tilde F_0 (\bb  ,y)| \leq C\frac{(1-\bb^2)^{\frac{2-p}{p-1}}}{(1+\bb y)^{\frac{p+1}{p-1}}} .$\\
\end{tabular} 
\end{center}
Using this and Claim \ref{claim}, we see that (\ref{majoration1}) holds for $\tilde{F_0}$.
\end{proof}
\subsection{The conjugate operator $\tilde L_\bb ^*$}
In this step of the work of \cite{MR2362418}, the authors have computed  $\check L_\bb ^*$ by simple calculations using the definition of the conjugate, namely that $\phi(\check L_\bb(q),r)=\phi(q,\check  L_\bb ^*(r))$ and the fact that $\mathcal{L}$ is self-adjoint. By the same way, we introduce, in the following, the conjugate operator of $\tilde L_\bb $ with respect to $\phi$:
\begin{lem}{\bf (The conjugate operator of $\tilde L_\bb $ with respect to $\phi$)} For all $|\bb|< 1$, the operator $\tilde L_\bb ^*$ conjugate of $\tilde L_\bb $ with respect to $\phi$ is given by
\begin{equation}\label{L*_bb }
\tilde L_\bb ^*
\begin{pmatrix}
r_1\\
r_2\end{pmatrix}= \begin{pmatrix}
\tilde R_\bb (r_2)\\
-\mathcal{L} r_1+r_1+\frac{p+3}{p-1}r_2+2y r_2'-\frac{8}{(p-1)}\frac{r_2}{(1-y^2)}\end{pmatrix}
\end{equation}
for any $(r_1, r_2)\in (\mathcal{D}(\mathcal{L}))^2$, where $g=\tilde R_\bb  (r_2)$ is the unique solution of
\begin{equation}
-\mathcal{L} g+g=\mathcal{L} r_2+\tilde{\psi} (\bb , y)r_2. \label{sol}
\end{equation}
\end{lem}
\begin{proof} The proof is the same as the proof of Lemma 4.1 page 81 in \cite{MR2362418}.
\end{proof}
In the following, we give an eigenfunction of $\tilde L_\bb ^*$ associated to the eigenvalue $\lambda=0$.
\begin{lem}{\bf(Eigenfunction of $\tilde L_\bb ^*$ for the eigenvalue $\lambda=0$)} \label{3}
\item{(i)} (Existence) For all $|\bb|< 1$, there exists $\tilde W_0 \in \mathcal{H}$ continuous in terms of d such that $\tilde L_\bb ^* (\tilde W_0)=0$ where 
\begin{equation}\label{W}
\tilde W_{0, 2}(\bb  ,y)= \tilde c_0 \kappa(\bb , y) \mbox{ and }\frac{1}{\tilde c_0}=\frac{4\kappa_0^2}{p-1} \int_{-1}^1 \frac{\rho(y)}{1-y^2}dy
\end{equation}
and $\tilde W_{0, 1}$ is the unique solution of the equation
\begin{equation}\label{tildeW_1}
-\mathcal{L}g+ g =-\frac{p+3}{p-1} r_2- 2 y r'_2 + \frac{8}{p-1} \frac{r_2}{1-y^2}
\end{equation}
with $r_2= \tilde W_{0, 2}$. Moreover, we have
\begin{equation}\label{1}
\phi (\tilde W_0,\tilde{F_0})=1.
\end{equation}
\item{(ii)} {(Normalization)}
There exists $C_0 >0 $ such that for $|\bb|< 1$,
\begin{equation}\label{normalization}
||\tilde W_0(\bb,\cdot)||_{\mathcal{H}} \leq C_0 \;\mbox{  and  }\; ||\partial_\bb  \tilde W_0(\bb,\cdot)||_{\mathcal{H}} \leq \frac{C_0}{1-\bb^2}.
\end{equation}
\end{lem}
Before proving this Lemma, let us recall the result from \cite{MR2362418}.
 \begin{cl}\label{mimo}
  For any $r_2\in {\mathcal{H}_0}$, the equation (\ref{tildeW_1}) has a unique solution $g\in {\mathcal{H}_0}$ (\ref{10}) such that
$$||g||_{\mathcal{H}_0}\le C ||r_2||_{\mathcal{H}_0}.$$
 \end{cl}
\begin{proof} See Claim 4.5 page 86 in \cite{MR2362418}.
\end{proof}
\begin{proof}[Proof of Lemma \ref{3}]
\item {(i)} From the definition of $\tilde L_\bb ^*$ (\ref{L*_bb }), $\tilde W_{0}(\bb,\cdot)=(\tilde W_{0, 1}(\bb,\cdot),\tilde W_{0, 2}(\bb,\cdot))$ is an eigenfunction for the eigenvalue $\lambda=0$ if and only if
\begin{numcases}
 \,\tilde{R}_\bb (\tilde W_{0, 2}(\bb,\cdot))=0,  \label{R} \\
-\mathcal{L} \tilde W_{0, 1}(\bb,\cdot)+\tilde W_{0, 1}(\bb,\cdot)+ \frac{p+3}{p-1}\tilde W_{0, 2}(\bb,\cdot)+2y \partial_y \tilde W_{0, 2}(\bb,\cdot)-\frac{8}{(p-1)}\frac{\tilde W_{0, 2}(\bb,\cdot)}{(1-y^2)}=0.\label{eqtildewd}
\end{numcases}
Note that $\tilde R_\bb (r_2)$ is the unique solution of (\ref{sol}). Therefore, if $\tilde W_0$ is a solution of $(\ref{R})$-$(\ref{eqtildewd})$, then we have
$$\mathcal{L}\tilde W_{0, 2}(\bb,\cdot)+\tilde \psi(\bb,\cdot) \tilde  W_{0, 2}(\bb,\cdot)=0. $$
Note also that, since $\kappa(\bb , y)$ is a stationary solution of equation (\ref{equa}), it follows that
\begin{equation}
\mathcal{L}\kappa(\bb , y)+\tilde \psi \kappa(\bb , y)=0.
 \end{equation}
This suggests that we take $\tilde W_{0, 2}(\bb , y)=\tilde c_0(\bb )\kappa(\bb , y)$ with $\tilde c_0(\bb )  \ne0$ and $\tilde W_{0, 1}$ the unique solution of (\ref{eqtildewd}) (note that $\kappa(\bb ,\cdot) \in \mathcal{H}_0$ by definition (\ref{defk}) and use Claim \ref{mimo} for the existence and uniqueness of $\tilde W_{0, 1}(\bb,\cdot)$).
In this step, we will try to normalize $\tilde W_{0}$. From the definition of $\phi$ (\ref{phi}), Lemma $\ref{fonctions_propres}$ and (\ref{tildeW_1}), we write
\begin{eqnarray*}
\phi (\tilde W_0(\bb ,\cdot), \tilde F_0(\bb ,\cdot))&=&\int_{-1}^{1}((-\mathcal{L} \tilde W_{0, 1}(\bb,y)+\tilde W_{0, 1}(\bb,y)) \kappa(\bb ,y)+\tilde W_{0, 1}(\bb,y) \kappa(\bb ,y) )  \rho (y)dy \\
&=&\int_{-1}^{1}(-\frac{p+3}{p-1} \tilde W_{0, 2}(\bb,y)-2 y \tilde W_{0, 2}^{'}(\bb,y)+\frac{8}{p-1} \frac{\tilde W_{0, 2}(\bb,y)}{(1-y^2)}) \kappa(\bb ,y) \rho(y) dy.
\end{eqnarray*}
Note in particular that $\tilde W_{0, 2}(\bb ,y)= \tilde c_0(\bb )\kappa(\bb ,y) $, so
\begin{align*}
\phi &(\tilde W_0(\bb ,\cdot), \tilde F_0(\bb ,\cdot))\\&= \tilde c_0(\bb ) \left[\int_{-1}^{1}(-\frac{p+3}{p-1}+\frac{8}{(p-1)(1-y^2)} ) \kappa^2(\bb ,y) \rho(y) dy+\int_{-1}^{1}\kappa^2(\bb ,y)(y\rho(y))' dy \right] \\
&=\tilde c_0(\bb )\int_{-1}^{1}\left(-\frac{p+3}{p-1}+\frac{8}{(p-1)(1-y^2)}+1-\frac{4y^2}{(p-1)(1-y^2)}\right)\kappa^2(\bb ,y) \rho(y) dy\\
&= \tilde c_0(\bb )\frac{4}{p-1} \kappa_0^2(1-\bb^2)^\frac{2}{p-1} \int_{-1}^{1} \frac{1}{(1+\bb y)^{\frac{4}{p-1}}} \frac{\rho(y)}{1-y^2} dy.
\end{align*}
Performing the change of variable $Y=\frac{y+\bb }{1+\bb y}$, we get
\begin{eqnarray*}
\phi (\tilde W_0(\bb ,\cdot), \tilde F_0(\bb ,\cdot))=\tilde c_0(\bb )\frac{4}{p-1} \kappa_0^2 \int_{-1}^{1} (1-Y^2)^{\frac{3-p}{p-1}}\, dY.
\end{eqnarray*} Therefore, in order to get $\phi (\tilde W_0, \tilde F_0)=1,$ it is enough to fix $\tilde c_0(\bb )$ as a positive constant independent from $\bb$ as stated in (\ref{W}).
\item {(ii)\it (Normalization)} Since $\tilde W_{0,1}$ and $\partial_\bb  \tilde W_{0,1}$ are solutions to equation. (\ref{tildeW_1}) respectively with $r_2=\tilde W_{0,2}$ and $r_2=\partial_\bb  \tilde W_{0,2}$, we see from Claim \ref{mimo} that for all $|\bb|<1,$
\begin{equation}\label{128}
||\tilde W_0||_{\mathcal{H}} \leq C_0||\tilde W_{0,2}||_{\mathcal{H}_0} \;\mbox{  and  }\; ||\partial_\bb  \tilde W_0||_{\mathcal{H}} \leq C_0 ||\partial_\bb  \tilde W_{0,2}||_{\mathcal{H}_0}.
\end{equation}
Using (\ref{naro}) together with the definition of $\tilde W_{0,2}$ and straightforward computations, we see that for all $|\bb|<1$ and $|y|<1$,
\begin{center}
\begin{tabular}{ll}
  $|\tilde W_{0,2} (\bb  ,y)| \leq C\frac{(1-\bb^2)^\frac{1}{p-1}}{(1+\bb y)^\frac{2}{p-1}}, $&$ |\partial_y \tilde W_{0,2} (\bb  ,y)| \leq C\frac{(1-\bb^2)^\frac{1}{p-1}}{(1+\bb y)^{\frac{p+1}{p-1}}} ,$ 
\end{tabular}
\begin{tabular}{ll} 
 $|\partial_\bb  \tilde W_{0,2}(\bb  ,y)| \leq C\frac{(1-\bb^2)^{\frac{2-p}{p-1}}}{(1+\bb y)^\frac{2}{p-1}}, $&$ |\partial^2_{\bb ,y} \tilde W_{0,2}(\bb  ,y)| \leq C\frac{(1-\bb^2)^{\frac{2-p}{p-1}}}{(1+\bb y)^{\frac{p+1}{p-1}}} .$\\
\end{tabular}
\end{center}
Since we have by this, by Claim \ref{claim} and by the definition of the norm in ${\mathcal{H}_0},$ $||\tilde W_{0, 2}||_{\mathcal{H}_0}+(1-\bb^2)||\partial_\bb  \tilde W_{0, 2}||_{\mathcal{H}_0}\leq C_0,$ we see that (\ref{normalization}) follows by (\ref{128}). This concludes the proof of Lemma \ref{3}.
\end{proof}
\subsection{Expansion of $q$ with respect to the eigenspaces of $\tilde L_\bb $}
In the following, we expand any $r \in \mathcal{H}$ with respect to the eigenspaces of $\tilde L_\bb $ partially computed in Lemma \ref{fonctions_propres}. We claim the following:
\begin{definition}\label{4.6}
{\bf(Expansion of $r$ with respect to the eigenspaces of $\tilde L_\bb $)}. Consider $r \in \mathcal{H}$ and introduce 
\begin{equation} \label{pi}
\tilde \pi_0^\bb  (r)=\phi (\tilde W_0(\bb,\cdot), r), 
\end{equation}
where $\tilde W_0(\bb,\cdot)$ is the eigenfunction of $\tilde L_\bb ^*$ computed in Lemma \ref{3}, and $\tilde \pi_{-}^\bb (r) $ is defined by 
\begin{equation} \label{s}
r=\tilde \pi_0^\bb  (r) \tilde F_0(\bb  ,\cdot)+ \tilde \pi_{-}^\bb  (r).
\end{equation}
\end{definition}
Applying the operator $\tilde \pi_{0}^\bb $ to (\ref{s}), we write
$$\tilde \pi_{0}^\bb  (r)= \tilde \pi_{0}^\bb  (r)\tilde \pi_{0}^\bb  (\tilde F_0(\bb,\cdot))+\tilde \pi_{0}^\bb  (\tilde \pi_{-}^\bb  (r)).$$ 
By (\ref{1}), 
$$\tilde \pi_{0}^\bb  (\tilde F_0(\bb,\cdot))=\phi (\tilde W_0(\bb,\cdot),\tilde F_0(\bb,\cdot))=1,$$
therefore
\begin{equation*}
\phi (\tilde W_0(\bb,\cdot), \tilde \pi_{-}^\bb  (r))=\tilde \pi_{0}^\bb  (\tilde \pi_{-}^\bb  (r))=0.
\end{equation*}
Thus, we have
\begin{equation}\label{ss}
\tilde \pi_{-}^\bb  (r) \in \tilde{\mathcal{H}}_{-}^\bb  \equiv \{r \in \mathcal{H} \,|\,  \tilde \pi_{0}^\bb  (r)=0\}.
\end{equation}
{\bf Remark.} Note that if $r \in \tilde{\mathcal{H}}_{-}^\bb  ,$ then $\tilde \pi_{-}^\bb  (r)= r$ (just use (\ref{s}) and (\ref{ss})) and $\tilde L_\bb  r \in\tilde{ \mathcal{H}}_{-}^\bb  .$ Indeed, using the definition of $\pi_0^\bb $ (\ref{pi}), the definition of $\tilde L_\bb ^*$ and Lemma \ref{3}, we write $\tilde \pi_{0}^\bb  (\tilde L_\bb  r)=\phi (\tilde W_0(\bb,\cdot),\tilde L_\bb  r)=\phi (\tilde L^*_\bb \tilde W_0(\bb,\cdot), r)=0.$ Moreover $\tilde \pi_{-}^\bb  (\tilde F_0 (\bb,\cdot))=0$ (just use (\ref{s}) with $r=\tilde F_0 (\bb,\cdot)$ and (\ref{1})).

\noindent{\bf Remark.} Note that $\tilde \pi_{0}^\bb  (r)$ is the projection of $r$ on the eigenfunction of $\tilde L_\bb $ associated to $\lambda=0$, and $\tilde \pi_{-}^\bb  (r)$ is the negative part of $r$.
\subsection{Equivalent norms on $\mathcal{H}$ and $\tilde{\mathcal{H}}^\bb _{-}$ adapted to the dispersive structure}
We introduce 
\begin{eqnarray}\label{134}
 \tilde \varphi_{\bb } (q,r)&=&\int_{-1}^{1} (-\tilde \psi(\bb ,\cdot) q_1r_1+q_1' r_1'(1-y^2)+q_2 r_2 ) \rho dy,\notag\\
&=&\int_{-1}^{1} (- q_1(\mathcal{L}r_1+\tilde \psi(\bb ,\cdot)r_1)+q_2 r_2 ) \rho dy.\label{135}
\end{eqnarray}

\begin{prop}\label{4.7} {\bf(Equivalence in $\tilde{\mathcal{H}}^\bb _{-}$ of the $\mathcal{H}$ norm and the $\tilde \varphi_\bb $ norm)} There exists $C_0 >0$ such that for all $|\bb|<1,$ the following holds: \item{(i)} {\bf(Equivalence of norms in $\tilde{\mathcal{H}}^\bb _{-}$)} For all $r\in \tilde{\mathcal{H}}^\bb _{-},$
$$\frac{1}{C_0} ||r||_{\mathcal{H} }^2\le \tilde \varphi_\bb  (r,r)\le C_0 ||r||_\mathcal{H} ^2.$$
\item{(ii)} {\bf(Equivalence of norms in $\mathcal{H}$)} For all $r\in \mathcal{H},$
$$\frac{1}{C_0} ||r||_{\mathcal{H} }\le\left(|\tilde \pi_{0}^\bb  (r)| +\sqrt{ \tilde \varphi_\bb  (r_{-},r_{-})}\right)\le C_0 ||r||_\mathcal{H}\mbox{ where }r_{-}=\tilde \pi_{-}^\bb  (r) .$$ 
\end{prop}
We introduce for all $\epsilon >0$ 
\begin{eqnarray}\label{137}
  \tilde \varphi_{\bb ,\epsilon} (q,r)&=&\int_{-1}^{1} q_1 \left( -(1-\epsilon) \mathcal{L} r_1+\left( -(1-\epsilon)\tilde \psi(\bb ,y)-\epsilon  \frac{2 p (p+1)}{(p-1)^2}\frac{(1-\bb^2)}{(1+\bb y)^2} \right) r_1 \right) \rho dy \notag\\&+&(1-\epsilon)\int_{-1}^{1} q_2 r_2 \rho dy.
\end{eqnarray}
To prove this proposition, we use the following:
\begin{lem}\label{4.8}{\bf (Reduction of the proof of Proposition \ref{4.7})} There exists $\epsilon_0 \in (0,1)$ such that for all $|\bb|<1$ and $r\in \tilde{\mathcal{H}}^\bb _{-},$ $\varphi_{\bb ,\epsilon_0}(r,r)\ge0.$
\end{lem}
\begin{proof}[ Lemma \ref{4.8} implies Proposition \ref{4.7}]
As we have $|\tilde \psi (\bb ,y)| \le \frac{C}{1-y^2}$, we proceed exactly like in \cite{MR2362418} page 91.
\end{proof}
\begin{proof}[\it Proof of Lemma \ref{4.8}:] We proceed in 3 parts:\\
- In Part 1, we find an hyperplane of $\mathcal{H} $ where $\tilde \varphi_{\bb , \epsilon}$ is nonnegative.\\
- In Part 2, we find a straight line in $\mathcal{H} $, where $\tilde \varphi_{\bb , \epsilon}$ is negative and which is ``orthogonal'' to $\tilde{\mathcal{H}}^\bb _{-}$ with respect to $\tilde \varphi_{\bb , \epsilon}$.\\
- In Part 3, we proceed by contradiction and prove that $\tilde \varphi_{\bb , \epsilon}$ is nonnegative on $\tilde{\mathcal{H}}^\bb _{-}$.
\medskip\\
{\bf Part 1 : $\tilde \varphi_{\bb , \epsilon}$ is nonnegative on a hyperplane}\\
We claim the following:
\begin{lem}\label{4.9}
There exists $\epsilon_1 > 0$ such that for all $|\bb|<1$ and $\epsilon \in (0,\epsilon_1 ], $ $\tilde \varphi_{\bb ,\epsilon}$ is nonnegative on the hyperplane 
\begin{equation}\label{141}
 E_1=\left\{ q \in \mathcal{H} \big| \int_{-1}^{1} \mathcal{T}_{-\bb} (q_1)\rho (y) dy =0\right\},
\end{equation}
where $\mathcal{T}_{-\bb} $ is defined in (\ref{transformation}).
\end{lem}
\begin{proof}
Define from (\ref{26}) $\epsilon_1 = \min (1, \frac{\gamma_1}{\gamma_1-\frac{2 p (p+1)}{(p-1)^2}})>0$ and fix $\epsilon \in (0, \epsilon_1]$. We consider $u=(u_1, u_2)\in E_1,$ and write from (\ref{137})
\begin{eqnarray}\label{142}
\tilde \varphi_{\bb , \epsilon} (u,u)&= &\int_{-1}^{1} u_1 \left( -(1-\epsilon) \mathcal{L}u_1 +\left[ -(1-\epsilon) \tilde \psi (\bb , y)- \epsilon \frac{2 p (p+1)}{(p-1)^2} \frac{(1-\bb^2)}{(1+\bb y)^2}\right] u_1  \right) \rho (y) dy \notag\\
&+&(1-\epsilon) \int_{-1}^{1} u_2^2 \rho (y) dy.
\end{eqnarray}
If $U_1=\mathcal{T}_{-\bb}u_1,$ then $u_1=\mathcal{T}_{\bb }U_1$ and we have from (\ref{transformation})
\begin{eqnarray*}
u_1(y)=\frac{(1-\bb^2)^\frac{1}{p-1}}{(1+\bb y)^\frac{2}{p-1}} U_1 (z)\,\mbox{ with } z&=&\frac{y+\bb }{1+\bb y}, \notag\\
\mathcal{L}u_1(y) + \tilde \psi (\bb , y)u_1(y)&=& \frac{(1-\bb^2)^{\frac{p}{p-1}}}{(1+\bb y)^{\frac{2p}{p-1}}} \mathcal{L} U_1 (z), \notag\\
\rho (y) dy&=&\frac{(1+\bb y)^{\frac{2(p+1)}{p-1}}}{(1-\bb^2)^\frac{p+1}{p-1}} \rho (z) dz, \notag\\
0&=&\int U_1 (z)\rho (z) dz.
\end{eqnarray*}
Therefore, we see from (\ref{142}) and Lemma \ref{too} that
\begin{eqnarray*}
\tilde \varphi_{\bb , \epsilon} (u,u)&= &\int_{-1}^{1}  U_1 (z) \left( -(1-\epsilon ) \mathcal{L}U_1 (z)-\epsilon \frac{2 p (p+1)}{(p-1)^2}U_1 (z) \right) \rho (z) dz\notag\\
&+& (1-\epsilon )\int_{-1}^{1} u_2^2  \rho (y) dy.\\
&\ge& \left( -(1-\epsilon ) \gamma_1 -\epsilon \frac{2 p (p+1)}{(p-1)^2} \right) \int_{-1}^{1}  U_1^2 (z) \rho (z) dz+(1-\epsilon )\int_{-1}^{1} u_2^2  \rho (y) dy\ge 0
\end{eqnarray*}
since $\epsilon \le \epsilon_1$ hence $( -(1-\epsilon ) \gamma_1 -\epsilon \frac{2 p (p+1)}{(p-1)^2})\ge 0$ and $1-\epsilon \ge 0$. This concludes the proof of Lemma \ref{4.9}.
\end{proof}
{\bf Part 2 : $\tilde \varphi_{\bb , \epsilon}$ is negative on a straight line orthogonal to  $\tilde{\mathcal{H}}^\bb _{-}$.}\\
We need to find $\tilde V_0(\bb ,\epsilon,\cdot)$ in $\mathcal{H}$ such that $\tilde \varphi_{\bb , \epsilon}(\tilde V_0(\bb ,\epsilon,\cdot),r)=0$ for all $r \in \tilde{\mathcal{H}}^\bb _{-}.$ Since we know from the definition of $\tilde{\mathcal{H}}^\bb _{-}$ (\ref{ss}) that
$$\forall r \in \tilde{\mathcal{H}}^\bb _{-},\,\phi (\tilde W_0, r)=\tilde \pi_{0}^\bb  (r)=0,$$
we proceed as in page 93 in \cite{MR2362418} and search $\tilde V_0(\bb ,\epsilon,\cdot)$ such that
\begin{eqnarray}\label{144}
 \forall r \in \mathcal{H},\,\phi (\tilde W_0(\bb,\cdot), r)=\tilde \varphi_{\bb , \epsilon}(\tilde V_0(\bb ,\epsilon,\cdot),r).
\end{eqnarray}
Then, we will show that $\tilde \varphi_{\bb , \epsilon}$ is negative on the straight line spanned by $\tilde V_0(\bb ,\epsilon,\cdot)$. Consider $\epsilon >0$ small enough and take $|\bb|<1$. We claim the following:
\begin{lem}\label{4.10}
 There exists $\epsilon_2 >0$ such that for all $\epsilon \in(0, \epsilon_2)$ and $|\bb|<1$:\\
(i) There exists $\tilde V_{0}(\bb ,\epsilon,\cdot)\in \mathcal{H}_0$ such that (\ref{144}) holds.\\
(ii) Moreover there exists $ c>0$ such that
\begin{equation}\label{145}
 \sup_{|\bb|<1}||\epsilon \tilde V_{0}(\bb ,\epsilon,\cdot)+c \tilde F_{0}(\bb,\cdot) ||_{\mathcal{H} }\rightarrow 0\mbox{ as }\epsilon \rightarrow 0^+.
\end{equation}
(iii) The bilinear form $\tilde \varphi_{\bb , \epsilon}$ is negative on a line of $\mathcal{H}$ spanned by $\tilde V_0(\bb ,\epsilon,\cdot)$.
\end{lem}
\begin{proof}[ Proof of Lemma \ref{4.10}:] We proceed in 3 steps:\\
-In Step 1, we find a PDE satisfied by $\tilde V_{0,1}(\bb ,\epsilon,\cdot)$ and transform it with the Lorentz transform in similarity variables defined in (\ref{transformation}).\\
-In Step 2, we solve the transformed PDE and find the asymptotic behavior of $\tilde V_{0,1}(\bb ,\epsilon,\cdot)$ as $\epsilon \rightarrow 0^+$, uniformly in $|\bb|<1$, which gives (i) and (ii).\\
-In Step 3, we use that asymptotic behavior to show that $\tilde\varphi_{\bb ,\epsilon}$ is negative on a straight line spanned by $\tilde V_{0}(\bb ,\epsilon,\cdot)$, which gives (iii).

\bigskip
{\bf Step 1: Reduction to the solution of some PDE.}\\
From the definition of $\tilde \varphi_{\bb ,\epsilon}$ (\ref{137}) and $\phi$ (\ref{phi}), we see that in order to satisfy (\ref{144}), it is enough to take
\begin{equation}\label{146}
\tilde V_{0,2}(\bb ,\epsilon,\cdot)= \tilde W_{0,2}(\bb,\cdot)/ (1-\epsilon)
\end{equation}
and to prove the existence of $\tilde V_{0,1}(\bb ,\epsilon,\cdot)$ solution to
\begin{eqnarray}\label{147}
 -(1-\epsilon)\mathcal{L}\tilde V_{0,1}(\bb ,\epsilon,\cdot)&+&\left(  -(1-\epsilon)\tilde \psi (\bb , \cdot)- \epsilon \frac{2 p (p+1)}{(p-1)^2} \frac{(1-\bb^2)}{(1+\bb y)^2}\right)\tilde V_{0,1}(\bb ,\epsilon,\cdot)\notag\\&=&-\mathcal{L} \tilde W_{0,1}(\bb , \cdot)+\tilde W_{0,1}(\bb , \cdot).
\end{eqnarray}
\begin{cl}\label{4.11}{\bf (Reduction to an explicitly solvable PDE)} Consider $\tilde V_{0,1}(\bb ,\epsilon,\cdot)$ and introduce $\tilde v_{0,1}(\bb ,\epsilon,\cdot)$ defined by
\begin{equation}\label{148}
 \tilde v_{0,1}(\bb ,\epsilon,\cdot)=\mathcal{T}_{-\bb} \tilde V_{0,1}(\bb ,\epsilon,\cdot),
\end{equation}
where $\mathcal{T}_{\bb }$ is defined in (\ref{transformation}). Then,\\
(i) $\tilde V_{0,1}(\bb ,\epsilon,\cdot)$ is a solution to (\ref{147}) if and only if $\tilde v_{0,1}(\bb ,\epsilon,\cdot)$ is a solution to the equation
\begin{equation}\label{149}
 (1-\epsilon)\mathcal{L}\tilde v_{0,1}(\bb ,\epsilon,z)+\epsilon \frac{2p(p+1)}{(p-1)^2}\tilde v_{0,1}(\bb ,\epsilon,z)=f_0^\bb (z)\equiv\frac{1-\bb^2}{(1-\bb z)^2}\mathcal{T}_{-\bb} \left( \mathcal{L}  \tilde W_{0,1}(\bb,y)-\tilde W_{0,1}(\bb,y)\right).
\end{equation}
(ii) The linear form $h \mapsto \int_{-1}^{1} f_0^\bb  h \rho$ defined for all $h\in \mathcal{H}_0$ is continuous and for some $C_0 >0$, we have 
$$\forall \bb \in (-1,1),\, ||f_0^\bb ||_{\mathcal{H}_{0}'}\le C_0|| \tilde W_0(\bb,\cdot)||_{\mathcal{H}}\le C_0^2.$$
\end{cl}
\begin{proof}
 \item{(i)} Using (\ref{transformation}) we see that
\begin{eqnarray*}
\tilde V_{0,1}(\bb ,\epsilon,y)=\frac{(1-\bb^2)^\frac{1}{p-1}}{(1+\bb y)^\frac{2}{p-1}} \tilde v_{0,1}(\bb ,\epsilon,z)\,\mbox{ with } z&=&\frac{y+\bb }{1+\bb y}, \notag\\
\mathcal{L}\tilde V_{0,1}(\bb ,\epsilon,y) + \tilde \psi (\bb , y)\tilde V_{0,1}(\bb ,\epsilon,y)&=& \frac{(1-\bb^2)^{\frac{p}{p-1}}}{(1+\bb y)^{\frac{2p}{p-1}}} \mathcal{L} \tilde v_{0,1}(\bb ,\epsilon,z).
\end{eqnarray*}
Since $\frac{(1-\bb z)^2}{1-\bb^2}=\frac{1-\bb^2}{(1+\bb y)^2}$, we see that equation $(\ref{147})$ and $(\ref{149})$ are equivalent.
\item{(ii)}  The proof of (ii) is the same as the proof of Claim 4.11 in page 94 in \cite{MR2362418}. 
\end{proof}
\bigskip
{\bf Step 2: Solution of equation (\ref{149}) and asymptotic behavior as $\epsilon \rightarrow 0^+$.\\ }\\
We prove (i) and (ii) of Lemma \ref{4.10} here. Let us first recall the following result from \cite{MR2362418}.
\begin{cl}\label{4.12}{\bf (Solution of equation (\ref{149}))} Consider
$$f(y)=\sum \limits_{n=0}^\infty \tilde f_n h_n (y) \in \mathcal{H}_{0}'$$
where $h_n $ are the eigenfunctions of $\mathcal{L}$ defined in Proposition \ref{L}. Then, for any $\epsilon \in (0, \frac{1}{2})$, the following equation
\begin{equation*}
 (1-\epsilon)\mathcal{L} v+\epsilon \frac{2p(p+1)}{(p-1)^2}v=f
\end{equation*}
has a unique solution in $\ \mathcal{H}_0$ given by 
\begin{equation*}
 v=\sum \limits_{n=0}^\infty \frac{\tilde f_n}{\gamma_n+\left(\frac{2(p+1)}{(p-1)^2} -\gamma_n\right)\epsilon}  h_n
\end{equation*}
where $\gamma_n \le 0$ are the eigenvalues of $\mathcal{L}$ introduced in Proposition \ref{L}.
\end{cl}
\noindent Now, we use this Claim to prove $(i)$ and $(ii)$.\\
{\it Proof of (i) of Lemma \ref{4.10}:} Using $(ii)$ in Claim \ref{4.11}, we see that $f_0^\bb \in \mathcal{H}_{0}'$. Therefore, Claim \ref{4.12} applies, and we have a unique solution $\tilde v_{0,2}(\bb ,\epsilon,\cdot)\in \mathcal{H}_0$ to equation (\ref{149}). Using $(i)$ of Claim \ref{3} and Lemma \ref{ff} below we get a solution $\tilde V_{0,1}(\bb ,\epsilon,\cdot)\in \mathcal{H}_0$ to equation (\ref{147}). \\
{\it Proof of (ii) of Lemma \ref{4.10}}: Note that the spectral properties of $\mathcal{L}$ are given in Proposition \ref{L} below. Since $h_0=c_0$ by Proposition \ref{L}, we see from Claim \ref{4.12} and (ii) in Claim \ref{4.11} that for $\epsilon$ small enough,
\begin{equation}\label{158}
 ||\tilde v_{0,1}(\bb ,\epsilon,\cdot)-\frac{\tilde f_0}{\frac{2(p+1)}{(p-1)^2}\epsilon}c_0||_{\mathcal{H}_{0}}\le C ||f_0^\bb  ||_{\mathcal{H}_{0}'}\le C,
\end{equation}
where from $(ii)$ in Lemma \ref{ff} and the fact that $c_0=\mathcal{T}_{-\bb} (c_0\frac{\kappa(\bb ,y)}{\kappa_0})$ (see \ref{transformation}), we have
\begin{align*}
\tilde f_0^\bb  =c_0 \int_{-1}^{1} f(z) \rho (z)dz&= \frac{c_0 }{\kappa_0} \int_{-1}^{1} \left( \mathcal{L}  \tilde W_{0,1}(\bb  ,y)-\tilde W_{0,1}(\bb  ,y) \right)\kappa(\bb ,y) \rho (y)dy\\&=-\frac{c_0}{\kappa_0}\phi(\tilde W_0(\bb ,\cdot),\tilde F_0(\bb,\cdot))=-\frac{c_0}{\kappa_0}.
\end{align*}
(use also the expression (\ref{phi}) of $\phi$ together with (\ref{F}) and (\ref{1}).
As $\tilde V_{0,2}(\bb ,\epsilon,\cdot)$ is explicitly given by (\ref{146}) and (\ref{W}), we see that (\ref{145}) follows from (\ref{158}), (\ref{148}), the fact that $\mathcal{T}_\bb  (\kappa_0) =\kappa (\bb , y)$ and the expression of $\tilde F_{0}$ (\ref{F}).

\bigskip
{\bf Step 3: Sign of $\tilde \varphi_{\bb ,\epsilon}$ on the line spanned by $\tilde V_{0}(\bb ,\epsilon,\cdot)$.}\\
% We prove $(iii)$ of Lemma \ref{4.10} here.\\
{\it Proof of (iii) of Lemma \ref{4.10}:} We will prove now that $ \tilde \varphi_{\bb ,\epsilon} $ is negative on the straight line spanned by $\tilde V_{0}(\bb ,\epsilon,\cdot)$.
From (\ref{144}), (\ref{145}) and (\ref{1}), we see that %(\ref{1})
\begin{eqnarray*}
 \tilde \varphi_{\bb ,\epsilon} (\tilde V_{0}(\bb ,\epsilon,\cdot),\tilde V_{0}(\bb ,\epsilon,\cdot))=\phi (\tilde W_{0}(\bb,\cdot),\tilde V_{0}(\bb ,\epsilon,\cdot))\sim -\frac{c}{\epsilon} \phi (\tilde W_{0}(\bb,\cdot),\tilde F_{0}(\bb,\cdot))=-\frac{c}{\epsilon},\mbox{ as }\epsilon \rightarrow 0,
\end{eqnarray*}
uniformly in $|\bb|<1$, So $\tilde \varphi_{\bb ,\epsilon} (\tilde V_{0}(\bb ,\epsilon,\cdot),\tilde V_{0}(\bb ,\epsilon,\cdot)) <0$. This concludes the proof of Lemma \ref{4.10}.

\end{proof}
{\bf Part 3: End of the proof of Lemma \ref{4.8}:}\\
From Lemmas \ref{4.9} and \ref{4.10}, we define $\epsilon_0=\min(\epsilon_1, \epsilon_2)\in (0,1).$ We will now prove by contradiction that $\tilde \varphi_{\bb ,\epsilon_0}$ is positive on $\tilde{\mathcal{H}}^\bb _{-}$ for all $|\bb|<1.$\\
We note that from (\ref{ss}) and (\ref{144}), for all $|\bb|<1$ and $\epsilon \in (0,\epsilon_0]$, the definition of $\tilde{\mathcal{H}}^\bb _{-}$ (\ref{ss}) writes as follows:
\begin{equation}\label{164}
 \tilde{\mathcal{H}}_{-}^\bb  = \{r \in \mathcal{H} \,|\,  \tilde \varphi_{\bb ,\epsilon}(\tilde V_{0}(\bb ,\epsilon,\cdot),r)=0\}.
\end{equation}
Consider $|\bb|<1$. By contradiction, assume that $\tilde \varphi_{\bb ,\epsilon}$ is negative so
\begin{equation}\label{165}
 \mbox{there is a nonzero }r\in \tilde{\mathcal{H}}^\bb _{-} \mbox{ such that }\tilde \varphi_{\bb ,\epsilon}(r,r)<0.
\end{equation}
We mention that $r$ is not collinear $\tilde V_{0}(\bb ,\epsilon,\cdot)$. Indeed, if $r =\alpha \tilde V_{0}(\bb ,\epsilon,\cdot)$ with $\alpha \in \mathbb{R^*}$, then we would have
$$ \tilde \varphi_{\bb ,\epsilon}(\tilde V_{0}(\bb ,\epsilon,\cdot),r)=\alpha  \tilde \varphi_{\bb ,\epsilon}(\tilde V_{0}(\bb ,\epsilon,\cdot),\tilde V_{0}(\bb ,\epsilon,\cdot))\neq0,$$
by $(iii)$ in Lemma \ref{4.10}, which contradicts (\ref{164}). Thus, the vector subspace
$$E_2=\mbox{ span }(\tilde V_{0}(\bb ,\epsilon,\cdot),r)$$
is of dimension 2. Therefore, as the subspace $E_1$ (\ref{141}) is of codimension 1, there exists a non zero $u \in E_1 \cap E_2.$\\
On the one hand, since $u \in E_1$, we have from Lemma \ref{4.9} that
\begin{equation}\label{166}
 \tilde \varphi_{\bb ,\epsilon}(u,u) \ge 0.
\end{equation}
On the other hand, since $ \tilde \varphi_{\bb ,\epsilon}$ is negative on $E_2$ by (iii) of Lemma \ref{4.10}, we must have from (\ref{164}) and (\ref{165}), 
\begin{equation*}
 \tilde \varphi_{\bb ,\epsilon}(u,u) < 0.
\end{equation*}
This contradicts (\ref{166}). So, $\tilde \varphi_{\bb ,\epsilon}$ is nonnegative on $\tilde{\mathcal{H}}^\bb _{-}$. This concludes the proof of Lemma \ref{4.8} and Proposition \ref{4.7}.
\end{proof}
\section{Trapping near the set of stationary solutions}
In this part of the work in the real case in \cite{MR2362418}, the authors have assumed that (\ref{18}) holds for some $s^*\in\mathbb{R}$ and $d^*\in (-1,1)$ and use modulation theory to introduce a parameter $d(s)$ adapted to the linearized equation and derive from the energy barrier the smallness of the unstable direction with respect to the stable, then they use this to show that $(w(s),\partial_s w(s))$ to some $\kappa(\bb_\infty,\cdot)$ as $s \rightarrow \infty$ in the norm of $\mathcal{H}$.

This section is devoted to the proofs of Proposition \ref{2}, Theorem \ref{theo3} and Theorem \ref{theo4}. Let us first give the proof of Proposition \ref{2} then derive Theorem \ref{theo4} from Theorem \ref{theo3}, and afterwards, prove Theorem \ref{theo3}.
\subsection{Convergence to a stationary solution}
We give the proofs of Proposition \ref{2} and Theorem \ref{theo4} here.
\begin{proof}[Proof of Proposition \ref{2}]
From Proposition \ref{2.1} and Proposition \ref{Th}, one can see that the proof given in the real case in Section 3.1 in \cite{MR2362418} holds here with non change. Indeed, all the estimates remain valid in the complex case, in particular, the Sobolev embedding and the Duhamel formulation of the wave equation (\ref{waveq}).
\end{proof}
\begin{proof}[Proof of Theorem \ref{theo4} assuming Theorem \ref{theo3}]
 Consider $w=w_{x_0} $ where $x_0$ is non-characteristic. The conclusion will follow from the application of Theorem \ref{theo3} to $w_{x_0}.$ In order to conclude, we have to check conditions (\ref{17}) and (\ref{18}). From the monotonicity of functional $E$ (See Proposition \ref{2.1} below) and $(ii)$ of Proposition \ref{2}, we see that
\begin{align*}
 \forall s \ge -\log(T(x_0)),\,E(w(s),\partial_s w(s))\ge E(\kappa_0,0)
\end{align*}
and (\ref{17}) follows. Consider $\epsilon^*$ defined in Theorem \ref{theo3}. From (i) of Proposition \ref{2}, we have the existence of $s^*\ge -\log T(x_0)$ such that 
 \begin{eqnarray*}
\inf_{\{|\bb|<1,\, \theta \in \mathbb{R}\}}\Big|\Big|\begin{pmatrix} w(s^*)\\\partial_s w(s^*) \end{pmatrix} -e^{i \theta}\begin{pmatrix} \kappa(\bb ,\cdot)\\0\end{pmatrix} \Big|\Big|_{{ H^1\times L^2}}\le \frac{\epsilon^*}{2}.
 \end{eqnarray*}
Therefore, there exists $|\bb^*|<1$ and $\theta^*\in \mathbb{R}$ such that
 \begin{eqnarray*}
\Big|\Big|\begin{pmatrix} w(s^*)\\\partial_s w(s^*) \end{pmatrix} -e^{i \theta^*}\begin{pmatrix} \kappa(\bb ^*,\cdot)\\0\end{pmatrix} \Big|\Big|_{{ H^1\times L^2}}\le \epsilon^*.
 \end{eqnarray*}
 Since $0\le \rho(y)\le 1$, it follows that
 \begin{eqnarray*}
\inf_{\{|\bb|<1, \theta \in \mathbb{R}\}}\Big|\Big|\begin{pmatrix} w(s^*)\\\partial_s w(s^*) \end{pmatrix} -e^{i \theta}\begin{pmatrix} \kappa(\bb ,\cdot)\\0\end{pmatrix} \Big|\Big|_{\mathcal{H}}\le \epsilon^*
 \end{eqnarray*}
 and (\ref{18}) follows. Applying Theorem \ref{theo3}, we get the conclusion of the Theorem \ref{theo4}.
\end{proof}

\subsection{A Modulation technique}
We introduce two parameters $\bb(s)$ and $\theta(s)$ and we use a modulation technique to claim the following:
\begin{prop}{({\bf Modulation of $w$ with respect to $e^{i\theta}\kappa (\bb , \cdot)$)}}\label{5.1}
There exists $\epsilon_	1 > 0$ and $K_1 >0$ such that if $(w, \partial_s w) \in C([s^*,\infty),\mathcal{H})$ for some $s^* \in \mathbb{R}$ is a solution to equation (\ref{equa}) which satisfies (\ref{18}) for some $|\bb^*|<1, \theta^* \in \mathbb{R}$ and $\epsilon^* <\epsilon_1$, then the following is true:
\item{(i)} {\bf (Choice of the modulation parameter)} There exists $\bb(s)\in C^1([s^*,\infty),(-1,1))$ and $\theta(s)\in C^1([s^*,\infty),\mathbb{R}) $ such that for all $s\in [s^*,\infty)$,
\begin{equation}\label{167}
\check\pi_0^{\bb(s)}(\check q(s))=\tilde\pi_0^{\bb(s)}(\tilde q(s))=0
\end{equation}
where $\check\pi_0^{\bb}$ and $\tilde\pi_0^{\bb}$ are defined in (\ref{barpi}), (\ref{pi}) and $q=(q_1, q_2)$ is defined for all $s\in [s_0,\infty)$ by
\begin{eqnarray}\label{168}
\begin{pmatrix} w(y,s)\\\partial_s w(y,s) \end{pmatrix} =e^{i \theta (s)}\left[\begin{pmatrix} \kappa(\bb (s),y)\\0\end{pmatrix} +\begin{pmatrix} q_1(y,s)\\q_2(y,s)\end{pmatrix}\right].
\end{eqnarray}
Moreover,
\begin{equation*}
 |\theta(s^*)-\theta^*|+\Big| \log\left(\frac{1+\bb (s^*)}{1-\bb(s^*)}\right)-\log\left(\frac{1+\bb^*}{1-\bb^*}\right)\Big|+||q(s^*)||_{\mathcal{H}}\le K_1 \epsilon^*.
\end{equation*}
\item{(ii)} {\bf (Equation on $q$)} For all $s\in [s^*,\infty)$,
\begin{align}\label{170}
\frac{\partial }{\partial s}\begin{pmatrix} \check q_1\\\check q_2
\end{pmatrix}&=\check L_{\bb(s)} \begin{pmatrix} \check q_1\\\check q_2 \end{pmatrix}+\begin{pmatrix}0\\ \check{f}_{\bb(s)}(q_1)\end{pmatrix}-\bb'(s)\begin{pmatrix}\partial_\bb  \kappa(\bb ,y)\\ 0\end{pmatrix}+\theta'(s)\begin{pmatrix} \tilde q_1 \\\tilde q_2 \end{pmatrix},
\\\label{170'}
\frac{\partial }{\partial s} \begin{pmatrix} \tilde q_1\\\tilde q_2
\end{pmatrix}&=\tilde L_{\bb(s)}\begin{pmatrix} \tilde q_1\\\tilde q_2 \end{pmatrix}+\begin{pmatrix} 0\\\tilde{f}_{\bb(s)}(q_1) \end{pmatrix}-\theta'(s)\begin{pmatrix} \kappa(\bb ,y)+\check q_1\\\check q_2
\end{pmatrix}, 
\end{align}
where $\check L_{\bb(s)} ,\tilde L_{\bb(s)}, \check{f}_{\bb(s)}$ and $\tilde{f}_{\bb(s)}$ are defined in (\ref{barL_bb }), (\ref{barL_bb }) and (\ref{tildeL_bb }).
\end{prop}
\begin{proof}
(i) From (\ref{barpi}) (\ref{s}), we see that the condition (\ref{167}) becomes\\ $\Phi((w(s),\partial_s w(s)), \bb(s), \theta (s))=0 $ 
where $ \Phi  \in C(\mathcal{H}\times (-1,1)\times\mathbb{R},\mathbb{R}\times \mathbb{R}) $ is defined by
\begin{equation}
\begin{array}{c}\label{172}
\Phi (v, \bb, \theta)=\begin{pmatrix}\check \Phi (v, \bb, \theta)\\\tilde \Phi (v, \bb, \theta)\end{pmatrix}=\begin{pmatrix}\phi(\mathcal{R}e (e^{-i\theta}v-(\kappa(\bb ,\cdot),0)),\check W_0)\\
\phi(\mathcal{I}m (e^{-i\theta}v),\tilde W_0)\end{pmatrix}
\end{array}
\end{equation}
We recall the following inequality which has been proved in page 102 in \cite{MR2362418}:
\begin{equation}\label{j}
 \forall \bb_1, \bb_2 \in (-1,1), ||\kappa(\bb _1,\cdot)-\kappa(\bb _2,\cdot)||_{\mathcal{H}_0}\le C_0 |\lambda_1-\lambda_2|\mbox{ where }\lambda_i=\log\left(\frac{1+\bb_i}{1-\bb_i}\right).
\end{equation}
We would like to apply the implicit function theorem to $\Phi$ near the point \\$(e^{i\theta^*}(\kappa(\bb ^*,\cdot),0),\bb^*,\theta^*)$. Three facts have to be checked :\\
1-First, note that
\begin{equation*}
\Phi (e^{i\theta^*}(\kappa(\bb ^*,\cdot),0),\bb^*,\theta^*)=0
\end{equation*}
2-Then, we compute from (\ref{172}), for all $u\in \mathcal{H}$,
\begin{equation*}
 D_v \check \Phi (v,\bb,\theta)(u)=\phi(\mathcal{R}e (e^{-i\theta}u),\check W_0),
\end{equation*}
\begin{equation*}
 D_v \tilde \Phi (v,\bb,\theta)(u)=\phi(\mathcal{I}m (e^{-i\theta}u),\tilde W_0) ,
\end{equation*}
so we have
\begin{equation}\label{175}
 ||D_v \check \Phi (v,\bb,\theta)||\le C_0 \mbox{ and } ||D_v \tilde \Phi (v,\bb,\theta)||\le C_0.
\end{equation}
3-Let $J(\check \Phi,\tilde \Phi )$ the jacobian matrix of $\Phi$, and $D$ its determinant so
\begin{equation*}
J= \left( \begin{array}{ll}
     \partial_\bb  \check \Phi&  \partial_\theta \check \Phi\\
     \partial_\bb    \tilde \Phi & \partial_\theta \tilde \Phi\end{array}  \right)
\end{equation*}
where
\begin{eqnarray*}
  \partial_\bb  \check \Phi&=&\phi((\partial_\bb  \kappa(\bb ,\cdot),0),\check W_0)+\phi(\mathcal{R}e (e^{-i\theta}v-(\kappa(\bb ,\cdot),0)), \partial_\bb  \check W_0)\\&=&\frac{2\kappa_0}{(p-1)(1-\bb^2)}+\phi(\mathcal{R}e (e^{-i\theta}v-(\kappa(\bb ,\cdot),0)), \partial_\bb  \check W_0)\\
 \partial_\theta \check \Phi&=&\phi(\mathcal{I}m (e^{-i\theta }v),\check W_0)\\
\partial_\bb  \tilde \Phi&=&\phi(\mathcal{I}m (e^{-i\theta }v),\partial_\bb  \tilde W_0)\\
\partial_\theta \tilde \Phi&=&\phi(-\kappa(\bb ,\cdot),\tilde W_0)+\phi(-\mathcal{R}e (e^{-i\theta}v)+\kappa(\bb ,\cdot),\tilde W_0)\\
&=&-1+\phi(-\mathcal{R}e (e^{-i\theta}v)+\kappa(\bb ,\cdot),\tilde W_0),
\end{eqnarray*}
referring to Lemma 4.4 in \cite{MR2362418} for the first equation and the orthogonality relation (\ref{1}) for the last one. Using The Cauchy-Schwarz inequality, the continuity of $\phi$ in $\mathcal{H}$, the bound (\ref{normalization}), Lemma 4.4 in  \cite{MR2362418},
and (\ref{j}), we see that if
\begin{equation}\label{174}
 |\theta-\theta^*|+\Big| \log\left(\frac{1+\bb }{1-d}\right)-\log\left(\frac{1+\bb^*}{1-\bb^*}\right)\Big|+||v-e^{i\theta^*}(\kappa(\bb ^*,\cdot),0)||_{\mathcal{H}}\le \epsilon_1
\end{equation}
for some $\epsilon_1>0$ small enough independent of $\bb^*$, then we have
\begin{eqnarray}\label{s1}
\big| \partial_\bb  \check \Phi&-&\frac{2\kappa_0}{(p-1)(1-\bb^2)} \big| \le \frac{C}{1-\bb^2} \bigg(
 \left||\kappa(\bb ^*,\cdot)-\kappa(\bb ,\cdot)\right||_{\mathcal{H}_0}\\
&+&||\mathcal{R}e (e^{-i\theta^*}v-(\kappa(\bb ^*,\cdot),0))||_\mathcal{H}+||\mathcal{R}e (v(e^{-i\theta}-e^{-i\theta^*}))||_\mathcal{H} \bigg)\le \frac{C\epsilon_1}{1-\bb^2}\notag,
\end{eqnarray}
\begin{eqnarray}\label{s2}
 |\partial_\bb  \tilde \Phi|\le \frac{C}{1-\bb^2}||\mathcal{I}m (e^{-i\theta }v)||_\mathcal{H}\le \frac{C\epsilon_1}{1-\bb^2},
\end{eqnarray}
\begin{eqnarray}\label{s3}
 |\partial_\theta \check \Phi|\le C ||\mathcal{I}m (e^{-i\theta }v)||_\mathcal{H}\le C\epsilon_1,
\end{eqnarray}
\begin{eqnarray}\label{s4}
 |\partial_\theta \tilde \Phi+1|&=&|\phi(-\mathcal{R}e (e^{-i\theta}v)+\kappa(\bb ,\cdot),\tilde W_0)|
\le C \bigg(||-\mathcal{R}e (e^{-i\theta^*}v)\\&+&\kappa(\bb ^*,\cdot)||_\mathcal{H}+||\mathcal{R}e (v(e^{-i\theta}-e^{-i\theta^*})) ||_\mathcal{H}+||\kappa(\bb ,\cdot)-\kappa(\bb ^*,\cdot)||_{\mathcal{H}_0}\bigg)\le C \epsilon_1.\notag
\end{eqnarray}
Collecting (\ref{s1})-(\ref{s4}), we see that for $\epsilon_1$ small enough, we have
\begin{equation}\label{175'}
 |D + \frac{2 \kappa_0}{(p-1)(1-\bb^2)}| \le \frac{C \epsilon_1}{1-\bb^2},
\end{equation}
 so we have the nondegeneracy of $\tilde \Phi$ near the point $(e^{i\theta^*}(\kappa(\bb ^*,\cdot),0),\bb^*,\theta^*)$. Applying the implicit function theorem, we see from (\ref{175}) and (\ref{175'}) that there exists $\epsilon_2,\, \epsilon_3 >0$, $C^1$ applications $ (f,g): \mathcal{H}\rightarrow (-1,1)\times \mathbb{R}$ such that for all $v \in \mathcal{H}$ satisfying $||e^{i\theta^*}(\kappa(\bb ^*,\cdot),0)-v||_\mathcal{H}\le \epsilon_2$ and for all $ (\bb ,\theta)\in (-1,1)\times \mathbb{R} $ satisfying $|\bb -\bb^*|+|\theta-\theta^*|\le \epsilon_3$ 
we have
\begin{equation}\label{fct implicite}
\Phi (v, \bb, \theta)=0 \Leftrightarrow (\bb ,\theta)=(f(v),g(v)).
\end{equation}
Take $\epsilon_0=\frac{\epsilon_2}{2}$ and consider $\epsilon^*\le \epsilon_0$. From (\ref{18}) and the continuity of $(w,\partial_s w),$
we see that for some $\sigma^* >s^*$, we have:
$$\forall s \in [s^*,\sigma^*], ||(w(s),\partial_s w (s))-e^{i \theta^*}(\kappa(\bb ^*,\cdot),0)||_\mathcal{H}\le 2 \epsilon^* \le \epsilon_2.$$ 
Therefore, from (\ref{172}) and (\ref{fct implicite}), we see that requiring (\ref{167}) is equivalent to have $d=f(w(s))$ and $\theta=g(w(s)).$ Since $f$ and $g$ are $C^1$, we get the conclusion with $C^1$ functions $\bb(s)$ and $\theta (s)$ such that (\ref{167}) holds for all $s \in [s^*,\sigma^*]$.\\
Now,  let's prove that $\sigma^*=+\infty$. By contradiction, suppose that $\sigma^*<+\infty$, we apply the implicit function theorem to $\Phi$ at the point $(v_n, \bb_n, \theta_n) \equiv ((w(s_n),\partial_s w (s_n)), \bb(s_n), \theta (s_n))$ where $s_n = \sigma^*-\frac{1}{n}$, and the uniform continuity  of $(w(s),\partial_s w(s))$ from $[\sigma^*-\eta_0,\sigma^*+\eta_0]$ to $\mathcal{H}$ for some $\eta_0>0$. In fact, from (\ref{fct implicite}), $\Phi ((w(s_n),\partial_s w (s_n)), \bb(s_n), \theta (s_n))=0 $, moreover (\ref{175}) and (\ref{175'}) are uniformly satisfied, so as above we see that we can define $\bb(s)$ for all $s \in[s_n,s_n+\epsilon_0]$ for some $\epsilon_0 >0$ independent of $n$. Therefore, for $n$ large enough, $\bb(s)$ exists beyond $\sigma^*$, which is a contradiction. Thus, $\sigma^*=+\infty.$

(ii) This is a direct consequence of the equation (\ref{equa}) satisfied by $w$ put in vectorial form:
\begin{eqnarray*}
 \partial_s w&=&v\\
\partial_s v&=&\mathcal{L}w-\frac{2(p+1)}{(p-1)^2}w+|w|^{p-1}w-\frac{p+3}{p-1} v- 2 y \partial_y v
\end{eqnarray*}
and the fact that $(\kappa(\bb ,\cdot),0)$ satisfies
$$\mathcal{L}\kappa(\bb ,\cdot)-\frac{2(p+1)}{(p-1)^2}\kappa(\bb ,\cdot)+|\kappa(\bb ,\cdot)|^{p-1}\kappa(\bb ,\cdot)=0$$
as a stationary solution. We have from (\ref{168})
\begin{eqnarray*}
 \partial_s q_1&= &q_2-i \theta' (\kappa(\bb (s),y)+q_1)-\bb'\partial_\bb  \kappa\\
 \partial_s q_2&=& \mathcal{L}\kappa(\bb (s),y)+\mathcal{L} q_1-\frac{2(p+1)}{(p-1)^2}(\kappa(\bb (s),y)+q_1)\\&+&|\kappa(\bb (s),y)+q_1|^{p-1}(\kappa(\bb (s),y)+q_1)-\frac{p+3}{p-1} q_2- 2 y \partial_y q_2-i\theta' q_2.
\end{eqnarray*}
Dissociating the real and the imaginary part of these equations, we get (\ref{170}) and (\ref{170'}).
\end{proof}
\subsection{Projection on the eigenspaces of the operator $L_\bb $}
 Given $s \ge s^*$ and following the previous section, we make in this section the following a priori estimate
\begin{equation}\label{179}
||q (s)||_\mathcal{H} \le \epsilon 
\end{equation}
for some $\epsilon >0$. From (\ref{167}), we will expand $\check q$ and $\tilde q$ respectively according to the spectrum of the linear operators $\check L_\bb $ and $\tilde L_\bb $ as in (\ref{ma}) and (\ref{s}):
\begin{align}\label{180}
\check q(y,s)&=\check \alpha_1 \check F_1(\bb  ,y)+ \check q_{-}(y,s)\\
\label{180*}
\tilde q(y,s)&=\tilde q_{-}(y,s)
\end{align}
where
\begin{equation}\label{181}
\check \alpha_1= \check\pi_1^{\bb(s)}(\check q),\;\check \alpha_0= \check\pi_0^{\bb(s)}(\check q)=0,\;\check \alpha_{-}(s)=\sqrt{\check\varphi_\bb  (\check q_{-},\check q_{-})}
\end{equation}
\begin{equation}\label{181'}
\tilde \alpha_0= \tilde\pi_0^{\bb(s)}(\tilde q)=0,\;\tilde\alpha_{-}(s)=\sqrt{\tilde\varphi_\bb  (\tilde q_{-},\tilde q_{-})}
\end{equation}
and
\begin{equation*}
\check q_{-}=\begin{pmatrix} \check q_{-,1}\\\check q_{-,2} \end{pmatrix}=\check\pi_{-}^{\bb}(\check q)=\check\pi_{-}^{\bb}\begin{pmatrix} \check{q_1}\\\check{q_2} \end{pmatrix}
\end{equation*}
\begin{equation*}
\tilde q_{-}=\begin{pmatrix} \tilde q_{-,1}\\\tilde q_{-,2} \end{pmatrix}=\tilde\pi_{-}^{\bb}(\tilde q)=\tilde\pi_{-}^{\bb}\begin{pmatrix} \tilde{q_1}\\\tilde{q_2} \end{pmatrix}
\end{equation*}
From (\ref{180}), (\ref{180*}), (\ref{36,5}) Proposition \ref{4.7}, we see that for all $s \ge s_0$,
\begin{eqnarray}\label{183}
\notag
 \frac{1}{C_0} \check \alpha_{-}(s) &\le& ||\check q_{-}(s) ||_{\mathcal{H} } \le C_0 \check \alpha_{-}(s)\\ 
 \frac{1}{C_0}(|\check \alpha_1(s)|+ \check \alpha_{-}(s)) &\le& ||\check q(s) ||_{\mathcal{H} } \le C_0 (|\check \alpha_1(s)|+\check \alpha_{-}(s))\\
 \frac{1}{C_0} \tilde \alpha_{-}(s) &\le& ||\tilde q(s) ||_{\mathcal{H} } \le C_0 \tilde \alpha_{-}(s) \notag
\end{eqnarray}
for some $C_0 > 0$. Let us introduce
\begin{equation}\label{209}
R_{-}(s)=-\int_{-1}^{1} \mathcal{ F}_{\bb }(q_1) \rho dy,
\end{equation}
where 
\begin{equation}\label{*F*}
\mathcal{F}_{\bb(s)} (q_1 (y,s))=\frac{|\kappa(\bb ,\cdot)+q_1|^{p+1}}{p+1}-\frac{\kappa(\bb ,\cdot)^{p+1}}{p+1}-\kappa(\bb ,\cdot)^p \check q_1-\frac{p}{2} \kappa(\bb ,\cdot)^{p-1} \check q_1^2-\frac{\kappa(\bb ,\cdot)^{p-1}}{2}\tilde q_1^2.
\end{equation}
 In the following proposition, we derive from (\ref{170}) and (\ref{170'}) differential inequalities satisfied by $\check \alpha_1(s)$, $\check \alpha_{-}(s)$, $\tilde \alpha_{-}(s)$, $\theta(s)$ and $\bb(s)$.
\begin{prop}\label{5.2}
There exists $C_0$ and $\epsilon_2 >0$ such that if $w$ a solution to equation (\ref{equa}) satisfying (\ref{167}) and (\ref{179}) at some time s for some $\epsilon \le \epsilon_2$, where $q$ is defined in (\ref{168}), then:
\item{(i)} {\bf (Control of the modulation parameter)}
\begin{equation}\label{184}
|\theta'|+\frac{|\bb '|}{1-\bb^2}\le C_0 (\check \alpha_1^2 +\check \alpha_{-}^2+\tilde \alpha_{-}^2).
\end{equation}
\item{(ii)} {\bf (Projection of equation (\ref{170}) on the different eigenspaces of $\check L_\bb $ and $\tilde L_\bb $)}
\begin{align}\label{185}
|\check \alpha_1' -\check \alpha_{1}|&\le C_0 (\check \alpha_1^2 +\check \alpha_{-}^2+\tilde \alpha_{-}^2),
\\\label{186}
\left( R_{-}+\frac{1}{2}(\check \alpha_{-}^2+\tilde \alpha_{-}^2)\right)'&\le -\frac{4}{p-1}\int_{-1}^{1}(\check q_{-,2}^2+\tilde q_{-,2}^2)\frac{\rho}{1-y^2}dy+ C_0 (\check \alpha_1^2 +\check \alpha_{-}^2+\tilde \alpha_{-}^2)^{\frac{3}{2}},
\end{align}
for $R_{-}(s)$, as defined in (\ref{209}), satisfying
\begin{equation}\label{187}
|R_{-}(s)|\le C_0 (\check \alpha_1^2 +\check \alpha_{-}^2+\tilde \alpha_{-}^2)^{\frac{1+\check p}{2}} \mbox{ where } \check p=\min(p,2)>1.
\end{equation}
\item{(iii)} {\bf (An additional relation)}
\begin{equation}\label{188}
\frac{d}{ds}\int_{-1}^{1} \check q_1 \check q_2 \rho \le -\frac{4}{5}\check \alpha_{-}^2+ C_0\int_{-1}^{1}\check q_{-,2}^2\frac{\rho}{1-y^2}+ C_0 (\check \alpha_1^2 +\tilde \alpha_{-}^2)
\end{equation}
\begin{equation}\label{188'}
\frac{d}{ds}\int_{-1}^{1} \tilde q_1 \tilde q_2 \rho \le -\frac{4}{5}\tilde \alpha_{-}^2+ C_0\int_{-1}^{1} \tilde q_2^2\frac{\rho}{1-y^2}+ C_0 (\check \alpha_1^2 +\check \alpha_{-}^2).
\end{equation}
\item{(iv)} {\bf (Energy barrier)} If moreover (\ref{17}) holds, then
\begin{equation}
 \check \alpha_1 (s)\le C_0 \check \alpha_-(s)+ C_1 \tilde \alpha_-(s).\label{189}
\end{equation}
\end{prop}
\noindent {\bf Remark:} Estimate (\ref{186}) shows a kind of Lyapunov functional for system (\ref{170})-(\ref{170'}). Indeed, if we imagine for a second that $\bb$ and $\theta$ do not depend on $s$ (in other words, if we forget the modulation technique), then proving (\ref{186}) reduces to finding a Lyapunov functional for system (\ref{170})-(\ref{170'}), which follows, as for equation (\ref{equa}), by multiplication by the conjugate of the time derivative, then, by integration over $(-1,1)$.\\
Because of the modulation, we need to be more careful and use $(i)$ to show that $|\bb '|$ and $|\theta'|$ are quadratic in $||q||_\mathcal{H}$.\\
\noindent{\bf Remark}: The estimates concerning $\theta' (s)$ and $ { R}_-$ are among the novelties of our paper, since they directly involve the complex structure. The other estimates are parallel to those of the real case treated in \cite{MR2362418}.\\
\noindent{\bf Remark}: The bahavior of the solution will be derived in Section \ref{4.4} below, thanks to the differential inequalities stated in Proposition \ref{5.2} above. One issue will be to show that the unstable direction $\check \alpha_1 $, which satisfies (\ref{185}) never dominates the other components. This fact is true from (\ref{189}), which is a direct consequence of the energy barrier hypothesis $ E(w(s),\partial_s w(s)) \ge E(\kappa_0,0)$ given in (\ref{17}). Let us stress the fact that such a hypothesis is natural, since Theorem \ref{theo3} will be applied with $w=w_{x_0}$ where $x_0$ is non-characteristic, and thanks to Proposition \ref{2}, we know that (\ref{189}) holds.

 Let us give the following estimate:
\begin{cl}\label{5.3}
For all $y\in (-1,1)$, $q_1=\check q_1+ i \tilde q_1$
\begin{equation}\label{190}
|f_{\bb(s)} (q_1 (y,s))|\le C_0 mM \left( k(\bb (s),y)^{p-2}|q_1 (y,s)|^2,|q_1(y,s)|^p\right),
\end{equation}
\begin{equation}\label{191}
|\mathcal{F}_{\bb(s)} (q_1 (y,s))|\le C_0 mM \left( k(\bb (s),y)^{p-2}|q_1 (y,s)|^3,|q_1(y,s)|^{p+1}\right),
\end{equation}
where $f_{\bb } (q_1 )$ and $\mathcal{F}_{\bb} (q_1 )$ are introduced in (\ref{barL_bb }) and (\ref{*F*}), $mM=\min$ if $1<p<2$ and $mM=\max$ if $p\ge 2$.
\end{cl}
\begin{proof}[Proof of Claim \ref{5.3}] Introducing $\xi=\check \xi+ i \tilde \xi=q_1/\kappa(\bb (s),y)$ and considering the cases $|\xi|\le1$ and $|\xi|\ge1$, for $f_\bb  (q_1)=\check f_\bb  (q_1)+i\tilde f_\bb (q_1)$ given in (\ref{barL_bb }), we get the conclusion.
\end{proof}
\begin{proof}[Proof of Proposition \ref{5.2}]
We proceed in 4 parts in order to prove Proposition \ref{5.2}:
\medskip\\
-In Part 1, we project equations (\ref{170}) and (\ref{170'}) respectively with the projectors $\check \pi_\lambda$ ($\lambda \in\{0,1\}$) (\ref{barpi}) and $\tilde\pi_0$ (\ref{pi}) and we derive the smallness condition on $\bb'$ and $\theta'$, together with (\ref{185}).\\
-In Part 2, we first give some preliminary estimations, then from the derivatives of $\check \alpha_{-}^2$ and $\tilde \alpha_{-}^2$ given by the quadratic form $\check \phi_\bb $ and $\tilde \phi_\bb $, we get (\ref{186}) and (\ref{187}).\\
-In Part 3, writing equations satisfied by $\check q$ (\ref{170}), $\tilde q$ (\ref{170'}) and using (\ref{167}) we prove (\ref{188}) and (\ref{188'}).\\
-In Part 4, we prove $(iv)$.
\bigskip\\
{\bf Part 1: Projection of equations (\ref{170}) and (\ref{170'})}\\
Projecting equation (\ref{170}) with the projector $\check \pi_\lambda^\bb $ (\ref{barpi}) for $\lambda=0$ and $\lambda=1$, we write
\begin{eqnarray*}
\check \pi_\lambda^\bb  (\partial_s \check{q})
=\check \pi_\lambda^\bb  (\check L_{\bb(s)} \check{q})+\check \pi_\lambda^\bb  \begin{pmatrix}0\\ \check{f}_{\bb(s)}(q_1)\end{pmatrix}-\bb'(s)\check \pi_\lambda^\bb  \begin{pmatrix}\partial_\bb  \kappa(\bb ,y)\\ 0\end{pmatrix}+\theta'(s)\check \pi_\lambda^\bb  \begin{pmatrix} \tilde{q_1}\\\tilde{q_2} \end{pmatrix},
\end{eqnarray*}
Proceeding exactly like in page 105 in \cite{MR2362418} with (\ref{179}) (\ref{183}), and using the fact that
\begin{equation*}
|\check \pi_\lambda^\bb  \begin{pmatrix} \tilde{q_1}\\\tilde{q_2} \end{pmatrix}|=|\phi (\check W_\lambda (\bb,\cdot) , \begin{pmatrix} \tilde{q_1}\\\tilde{q_2}\end{pmatrix})|\le ||\check W_\lambda (\bb,\cdot) ||_{\mathcal{ H}} ||\tilde{q}||_{\mathcal{ H}}\le C \tilde \alpha_{-},
\end{equation*}
we get
\begin{align}\label{A1}
\frac{2 \kappa_0}{(p-1)(1-\bb^2)} |\bb '|&\le \frac{C_0}{1-\bb^2}|\bb '|(|\check\alpha_1|+\check\alpha_{-})+C_0 (\check \alpha_1^2 +\check \alpha_{-}^2+\tilde \alpha_{-}^2)+C_0|\theta'|\tilde \alpha_{-}
\\\label{A2}
|\check\alpha_1'(s)-\check\alpha_1(s)|&\le\frac{C_0}{1-\bb^2}|\bb '|(|\check\alpha_1|+\check\alpha_{-})+C_0 (\check \alpha_1^2 +\check \alpha_{-}^2+\tilde \alpha_{-}^2)+C_0|\theta'|\tilde \alpha_{-}.
\end{align}
Now, projecting equation (\ref{170'}) with the projector $\tilde \pi_0^\bb $ (\ref{pi}), we get
\begin{eqnarray}\label{98,5}
\tilde \pi_0^\bb (\partial_s \tilde q)=\tilde \pi_0^\bb (\tilde L_{\bb(s)} \tilde q)+\tilde \pi_0^\bb  \begin{pmatrix} 0\\\tilde{f}_{\bb(s)}(q_1) \end{pmatrix}-\theta'(s)\tilde \pi_0^\bb \begin{pmatrix} \kappa(\bb ,y)+\check{q_1}\\\check{q_2} \end{pmatrix}.
\end{eqnarray}
-Since $\tilde \alpha_0 (s)=\tilde \pi_0^\bb  (\tilde q)= \phi ( \tilde W_0(\bb,\cdot),\tilde q)=0$ by (\ref{181'}) and the definition of $\tilde \pi_0^\bb $ (\ref{pi}), we write
$$0=\tilde \alpha_0' (s)=\tilde \pi_0^\bb  (\partial_s \tilde q)+ \bb'(s)\tilde \phi ( \partial_\bb  \tilde W_0(\bb,\cdot),\tilde q).$$
Using (\ref{normalization}) and (\ref{183}), we get
\begin{equation}\label{B1}
|\tilde \pi_0^\bb  (\partial_s \tilde q)|\le \frac{C_0}{1-\bb^2} |\bb '| \tilde\alpha_{-}.
\end{equation}
-Using (i) of Lemma \ref{3}, the definition of $\tilde \pi_0^\bb $ (\ref{pi}), we write
\begin{equation}\label{B2}
\tilde \pi_0^\bb  (\tilde L_{\bb }(\tilde q))=\tilde \phi ( \tilde W_0(\bb,\cdot),\tilde L_{\bb }(\tilde q))=\tilde \phi ( \tilde L_{\bb }^*(\tilde W_0(\bb,\cdot)),\tilde q)=0.
\end{equation}
-From the definitions of $\tilde \pi_0^\bb $ (\ref{pi}) and $\phi$ (\ref{phi}), together with Claim \ref{5.3}, we see that
\begin{align}
&\Big{|}\tilde \pi_0^\bb  \begin{pmatrix} 0\\\tilde{f}_{\bb(s)}(q_1) \end{pmatrix}\Big{|} \le C \int_{-1}^{1} \kappa(\bb ,y) |\tilde{f}_{\bb }(q_1)| \rho (y) dy \notag\\
&\le C_0 \int_{-1}^{1} \kappa(\bb ,y)^{p-1} |q_1(y,s)|^2\rho dy+ C_0 \delta_{\{p\ge2\}}  \int_{-1}^{1} \kappa(\bb ,y)|q_1(y,s)|^p\rho(y) dy\notag\\
&\le C_0 ||q_1||^2_{L_\rho^{p+1}} ||\kappa(\bb ,\cdot)||^{p-1}_{L_\rho^{p+1}}+C_0  \delta_{\{p\ge2\}}||q_1||^p_{L_\rho^{p+1}} ||\kappa(\bb ,\cdot)||_{L_\rho^{p+1}}\label{198}
\end{align}
where $\delta_{\{p\ge2\}}$ is $0$ if $1<p<2$ and $1$ otherwise. Therefore, using (\ref{198}), Lemma \ref{2.2}, (\ref{179}) and (\ref{183}), we get
\begin{equation}\label{B3}
\Big{|}\tilde \pi_0^\bb  \begin{pmatrix} 0\\\tilde{f}_{\bb(s)}(q_1) \end{pmatrix}\Big{|} \le C_0 (\check \alpha_1(s)^2 +\check \alpha_{-}(s)^2+\tilde \alpha_{-}(s)^2).
\end{equation}
-Since $\tilde \pi_0^\bb \begin{pmatrix} \kappa(\bb ,y)\\0 \end{pmatrix}=1$ from Lemma \ref{fonctions_propres} and \ref{3}, using (\ref{normalization}) and (\ref{183}), we write 
\begin{equation}\label{B4}
\left|\tilde \pi_0^\bb \begin{pmatrix} \kappa(\bb ,y)+\check q_1\\\check q_2 \end{pmatrix}-1\right|=\left|\tilde \pi_0^\bb \begin{pmatrix} \kappa(\bb ,y)+\check q_1\\\check q_2 \end{pmatrix}-\tilde \pi_0^\bb  \begin{pmatrix} \kappa(\bb ,y)\\0 \end{pmatrix}\right|\le C_0 (|\check\alpha_1(s)|+\check \alpha_- (s)).
\end{equation}
Using (\ref{B1}), (\ref{B2}), (\ref{B3}) and (\ref{B4}), to bound the terms of equation (\ref{98,5}) we get:
\begin{equation}\label{A3}
\left|\theta'\right|\le \frac{C_0}{1-\bb^2}|\bb '|\tilde\alpha_{-}+C_0 (\check \alpha_1^2 +\check \alpha_{-}^2+\tilde \alpha_{-}^2)+C_0|\theta'| (|\check\alpha_1|+\check\alpha_{-}).
\end{equation}

Using (\ref{179}) and (\ref{183}), we see that
$$|\theta'|+\frac{2 \kappa_0}{p-1} \frac{|\bb '|}{1-\bb^2}\le C_0 \epsilon \frac{|\bb '|}{1-\bb^2}+C_0 (\check \alpha_1^2 +\check \alpha_{-}^2+\tilde \alpha_{-}^2)+C_0 \epsilon |\theta'|,$$
hence
$$(1-C_0 \epsilon) |\theta'|+(\frac{2 \kappa_0}{p-1}-C_0 \epsilon ) \frac{|\bb '|}{1-\bb^2}\le C_0 (\check \alpha_1^2 +\check \alpha_{-}^2+\tilde \alpha_{-}^2).$$
Taking $\epsilon$ small enough, we get (\ref{184}). Then using (\ref{184}) to bound the term of the right hand side of (\ref{A2}) we get (\ref{185}).
\bigskip\\
{\bf Part 2: A kind of Lyapunov functional for system (\ref{170})-(\ref{170'})  }\\
We need to put together information from $\check q_-$ and $\tilde q_-$ in order to conclude. Handling each one alone doesn't allow to control the terms $\int_{-1}^{1} \check q_2 \check{f}_{\bb }(q_1) \rho dy $ and $\int_{-1}^{1} \tilde q_2 \tilde{f}_{\bb }(q_1) \rho dy $ which appear in the differential inequalities satisfied by $\check \alpha_-$ and $\tilde \alpha_-$ (use (\ref{y1}) and (\ref{y2}) below). We claim that (\ref{186}) follows from the following Lemmas:
\begin{lem}\label{5.40}{\bf (Preliminary estimates for $\check q_{-}$)} There exists $\epsilon_3 > 0$ such that if $\epsilon \le \epsilon_3$ in the hypotheses of Proposition \ref{5.2}, then
\begin{align*}
\Big{|}\Big{|}\partial_s \check q_{-} - \check L_\bb  (\check q_{-})-\check \pi_{-}^\bb   \begin{pmatrix} 0\\\check{f}_{\bb }(q_1) \end{pmatrix} \Big{|}\Big{|}_{\mathcal{ H}}&\le C_0 (\check \alpha_1^2 +\check \alpha_{-}^2+\tilde \alpha_{-}^2)^\frac{3}{2},
\\
\left| \check \varphi_\bb  \left(\check q_{-}, \check \pi_{-}^\bb   \begin{pmatrix} 0\\\check{f}_{\bb }(q_1) \end{pmatrix}\right)-\int_{-1}^{1} \check q_2 \check{f}_{\bb }(q_1) \rho dy \right| &\le C_0 (\check \alpha_1^2 +\check \alpha_{-}^2+\tilde \alpha_{-}^2)^\frac{3}{2}.
\end{align*}
\end{lem}
\begin{proof} Since the equation (\ref{170}) satisfied by $\check q$ is the same as in the real case treated in \cite{MR2362418}, except for the last term $\theta'(\tilde q_1,\tilde q_2)$, we refer the reader to Claim 5.4 page 106 in \cite{MR2362418}, and focus only on the last term. Using (\ref{183}) and (\ref{184}), we see that 
$$||\theta' (\tilde q_1,\tilde q_2)||_{\mathcal{H}}\le  C_0 (\check \alpha_1^2 +\check \alpha_{-}^2+\tilde \alpha_{-}^2)^\frac{3}{2},$$
which is precisely the error in the right-hand sides of the inequalities in Lemma \ref{5.40}.
Since our equation (\ref{equa}) satisfies (\ref{179}) and (\ref{183}) , we see that the proof of Merle and Zaag in page 108 in \cite{MR2362418} can be adapted in our case.
\end{proof}
\begin{lem}{\bf (Preliminary estimates for $ \tilde q_{-}$)} \label{5.4} There exists $\epsilon_4 > 0$ such that if $\epsilon \le \epsilon_4$ in the hypotheses of Proposition \ref{5.2}, then
\begin{align}\label{201}
\Big{|}\Big{|}\partial_s \tilde q_{-} - \tilde L_\bb  (\tilde q_{-})-\tilde \pi_{-}^\bb   \begin{pmatrix} 0\\\tilde{f}_{\bb }(q_1) \end{pmatrix} \Big{|}\Big{|}_{\mathcal{ H}}&\le C_0 (\check \alpha_1^2 +\check \alpha_{-}^2+\tilde \alpha_{-}^2)^\frac{3}{2},
\\\label{202}
\Big{|} \tilde \varphi_\bb  \left(\tilde q_{-}, \tilde \pi_{-}^\bb   \begin{pmatrix} 0\\\tilde{f}_{\bb }(q_1) \end{pmatrix}\right)-\int_{-1}^{1} \tilde q_2 \tilde{f}_{\bb }(q_1) \rho dy \Big{|} &\le C_0 (\check \alpha_1^2 +\check \alpha_{-}^2+\tilde \alpha_{-}^2)^\frac{3}{2},
\\\label{203}
\Big{|}\int_{-1}^{1} \check q_2 \check{f}_{\bb }(q_1) \rho dy +\int_{-1}^{1} \tilde q_2 \tilde{f}_{\bb }(q_1) \rho dy-\frac{d}{ds} \int_{-1}^{1}\mathcal{ F}_{\bb(s)} \rho dy  \Big{|} &\le C_0 (\check \alpha_1^2 +\check \alpha_{-}^2+\tilde \alpha_{-}^2)^2.
\end{align}
\end{lem}
\noindent{\bf Remark}: Note that (\ref{203}) is one of the new features of our paper. Indeed, it directly involves the complex structure.\\
Let us derive (\ref{186}) and (\ref{187}) from Lemmas \ref{5.40} and \ref{5.4}, then we prove Lemma \ref{5.4}.\\
\noindent{\it Proof of (\ref{186}) and (\ref{187})}: Using the definition of $\check \alpha_-$ (\ref{181}), we proceed like in page 107 in \cite{MR2362418} and we apply the bound (\ref{184}) on $|\bb '|$, so we get
\begin{equation*}
|\check \alpha_{-}'\check \alpha_{-}-\check \varphi_\bb (\check q_{-},\partial_s \check q_{-})|\le C_0 |\bb '|\frac{\check \alpha_{-}^2}{1-\bb^2}\le C_0  (\check \alpha_1^2 +\check \alpha_{-}^2+\tilde \alpha_{-}^2)^2.
\end{equation*}
Using Lemma \ref{5.40}, we write
\begin{align*}
&\Big{|}\check \alpha_{-}'\check \alpha_{-}-\check \varphi_\bb (\check q_{-},\check L_\bb  \check q_{-})-\int_{-1}^{1}\check q_2 \check{f}_{\bb }(q_1) \rho dy\Big{|} \notag\\
&\le C_0 (\check \alpha_1^2 +\check \alpha_{-}^2+\tilde \alpha_{-}^2)^\frac{3}{2}+\Big{|}\check \varphi_\bb \left(\check q_{-},\partial_s \check q_--\check L_\bb  (\check q_{-})-\check \pi_{-}^\bb   \begin{pmatrix} 0\\\check{f}_{\bb }(q_1) \end{pmatrix}\right)\Big{|}\notag\\
&\le C_0 (\check \alpha_1^2 +\check \alpha_{-}^2+\tilde \alpha_{-}^2)^\frac{3}{2}+||\check q_{-}||_{\mathcal{ H}} (\check \alpha_1^2 +\check \alpha_{-}^2+\tilde \alpha_{-}^2)^\frac{3}{2}\le C_0 (\check \alpha_1^2 +\check \alpha_{-}^2+\tilde \alpha_{-}^2)^\frac{3}{2}.
\end{align*}
 Since we easily gets from the definition (\ref{barL_bb }) of $\check L_\bb $ that  $\check \varphi_\bb  (\check q, \check L_\bb  (\check q))=-\frac{4}{p-1}\int_{-1}^{1} \check q_{-,2}^2 \frac{\rho}{1-y^2} dy $, we conclude that
\begin{equation}\label{y1}
|\check \alpha_{-}'\check \alpha_{-}-\int_{-1}^{1}\check q_2 \check{f}_{\bb }(q_1) \rho dy | 
\le- \frac{4}{p-1}\int_{-1}^{1}\check q_{-,2}^2 \frac{\rho}{1-y^2} dy +C_0 (\check \alpha_1^2 +\check \alpha_{-}^2+\tilde \alpha_{-}^2)^\frac{3}{2}.
\end{equation}
Arguing similarly for $\tilde \alpha_-$ (\ref{181}), we see that
\begin{equation*}
\tilde \alpha_{-}^2(s)=\tilde \varphi_\bb (\tilde q(s), \tilde q(s)).
\end{equation*}
Using the definition (\ref{134}) of $\tilde\varphi_\bb $, we have by differentiation
\begin{equation}\label{204}
\tilde \alpha_{-}'\tilde \alpha_{-}=\tilde \varphi_\bb (\tilde q,\partial_s \tilde q)-\frac{1}{2}\bb'(s)\int_{-1}^{1}\partial_\bb  \tilde\psi (\bb ,y)\tilde q_{1}^2\rho .
\end{equation}
Using the H\"{o}lder inequality, the Hardy-Sobolev estimate of Lemma \ref{2.2} and (\ref{183}), we write
  \begin{equation} \label{205}
  \Big| \int_{-1}^{1} \partial_\bb  \tilde \psi (\bb ,y)\tilde q_{1}^2\rho \Big|\le ||\partial_\bb  \tilde \psi (\bb ,y)||_{L_\rho^\frac{p+1}{p-1}}      ||\tilde q_{1}||_{L_\rho^{p+1}}^2\le C_0 ||\partial_\bb  \tilde \psi (\bb ,y)||_{L_\rho^\frac{p+1}{p-1}} \tilde \alpha_{-}(s)^2.
  \end{equation} 
Since $|\partial_\bb  \tilde \psi (\bb ,y)|\le C/(1+\bb y)^2$ for all $(\bb ,y) \in (-1,1)^2$ from the expression of $\tilde \psi$ in (\ref{barL_bb }), Using Claim \ref{claim}, we see that $||\partial_\bb  \tilde \psi (\bb ,y)||_{L_\rho^\frac{p+1}{p-1}}\le C/(1-\bb^2)$. Therefore, using (\ref{204}), (\ref{205}), and the bound (\ref{184}) on $|\bb '(s)|$, we get
\begin{equation}\label{206}
|\tilde \alpha_{-}'\tilde \alpha_{-}-\tilde \varphi_\bb (\tilde q,\partial_s \tilde q)|\le C_0 |\bb '|\frac{\tilde \alpha_{-}^2}{1-\bb^2}\le C_0  (\check \alpha_1^2 +\check \alpha_{-}^2+\tilde \alpha_{-}^2)^2.
\end{equation}
From (\ref{206}), the continuity of $\tilde \varphi_\bb  $, Lemma \ref{5.4}, we write
\begin{align}\label{207}
&\Big{|}\tilde \alpha_{-}'\tilde \alpha_{-}-\tilde \varphi_\bb (\tilde q,\tilde L_\bb  \tilde q)-\int_{-1}^{1}\tilde q_2 \tilde{f}_{\bb }(q_1) \rho dy\Big{|} \notag\\
&\le C_0 (\check \alpha_1^2 +\check \alpha_{-}^2+\tilde \alpha_{-}^2)^\frac{3}{2}+\Big{|}\tilde \varphi_\bb \left(\tilde q,\partial_s \tilde q-\tilde L_\bb  (\tilde q)-\tilde \pi_{-}^\bb   \begin{pmatrix} 0\\\tilde{f}_{\bb }(q_1) \end{pmatrix}\right)\Big{|}\notag\\
&\le C_0 (\check \alpha_1^2 +\check \alpha_{-}^2+\tilde \alpha_{-}^2)^\frac{3}{2}+||\tilde q||_{\mathcal{ H}} (\check \alpha_1^2 +\check \alpha_{-}^2+\tilde \alpha_{-}^2)^\frac{3}{2}\le C_0 (\check \alpha_1^2 +\check \alpha_{-}^2+\tilde \alpha_{-}^2)^\frac{3}{2}.
\end{align}
Besides, using the expressions of $\tilde L_\bb $ (\ref{tildeL_bb }) and $\tilde \varphi_\bb $ (\ref{135}), we have
 \begin{align}\label{208}
 &\tilde \varphi_\bb  (\tilde q, \tilde L_\bb  (\tilde q))= \tilde \varphi_\bb  \left(   \begin{pmatrix} \tilde{q_2}\\\mathcal{L}\tilde{q}_1+\tilde\psi (\bb , y)\tilde{q}_1-\frac{p+3}{p-1} \tilde{q}_2- 2 y \partial_y \tilde{q}_2 \end{pmatrix},
 \begin{pmatrix} \tilde{q}_1\\\tilde{q}_2\end{pmatrix} \right) \notag \\
 &= -\int_{-1}^{1}  \tilde{q}_2(\mathcal{L}\tilde{q}_1+\tilde\psi (\bb , y)\tilde{q}_1) \rho dy\notag \\
 &+ \int_{-1}^{1}  (\mathcal{L}\tilde{q}_1+\tilde\psi (\bb , y)\tilde{q}_1-\frac{p+3}{p-1} \tilde{q}_2- 2 y  \tilde{q}_2' )\tilde{q}_2 \rho dy\notag \\
 &= -\frac{p+3}{p-1}\int_{-1}^{1}  \tilde{q}_2^2 \rho dy-\int_{-1}^1 y (\tilde{q}_2^2)'\rho dy=-\frac{p+3}{p-1}\int_{-1}^{1}  \tilde{q}_2^2 +\int_{-1}^{1}  \tilde{q}_2^2 (\rho-y\rho')  dy\notag \\
 &=- \frac{4}{p-1}\left[ \int_{-1}^{1} \tilde q_{-,2}^2 \rho dy+\int_{-1}^{1} \tilde q_{-,2}^2 \frac{y^2 \rho}{1-y^2} dy \right]=-\frac{4}{p-1}\int_{-1}^{1} \tilde q_{-,2}^2 \frac{\rho}{1-y^2} dy .
 \end{align}
 Using (\ref{207}) and (\ref{208}), we see that
\begin{equation}\label{y2}
\left| \tilde \alpha_{-}' \tilde \alpha_{-}-\int_{-1}^{1} \tilde q_2  \tilde{f}_{\bb }(q_1) \rho dy \right| 
\le -\frac{4}{p-1}\int_{-1}^{1} \tilde q_{-,2}^2 \frac{\rho}{1-y^2} dy +C_0 (\check \alpha_1^2 +\check \alpha_{-}^2+\tilde \alpha_{-}^2)^\frac{3}{2}.
\end{equation}
Therefore, using (\ref{y1}) with (\ref{y2}), we write
\begin{align}\label{x1}
&\left| \frac{1}{2}(\check \alpha_{-}^2+\tilde \alpha_{-}^2)'-\left[ \int_{-1}^{1} \check q_2  \check{f}_{\bb }(q_1) +\tilde q_2  \tilde{f}_{\bb }(q_1)  \rho dy\right]\right | \\\notag
&\le -\frac{4}{p-1}\int_{-1}^{1} (\check q_{-,2}^2+\tilde q_{-,2}^2) \frac{\rho}{1-y^2} dy +C_0 (\check \alpha_1^2 +\check \alpha_{-}^2+\tilde \alpha_{-}^2)^\frac{3}{2}.
\end{align}
Better yet, by (\ref{203}) we see that estimate (\ref{186}) holds with $R_{-}$ given by (\ref{209}). Using Claim \ref{5.3}, Lemma \ref{2.2} and condition (\ref{179}) (considering first the case $p\ge 2$ and then the case $1<p<2$), we see that (\ref{187}) holds. It remains to prove Lemma \ref{5.4} in order to conclude the proof of (\ref{188}) and (\ref{189}).
\begin{proof}[Proof of Lemma \ref{5.4}]$ $\\
$\bullet$ {\it Proof of (\ref{201})}:
We first project equation (\ref{170'}) using the negative projector $\tilde \pi_-^\bb $ introduced in Definition \ref{4.6}
\begin{eqnarray}\label{naro1}
\tilde \pi_-^\bb (\partial_s \tilde q)=\tilde \pi_-^\bb (\tilde L_{\bb(s)} \tilde q)+\tilde \pi_-^\bb  \begin{pmatrix} 0\\\tilde{f}_{\bb }(q_1) \end{pmatrix}-\theta'(s)\tilde \pi_-^\bb \begin{pmatrix} \kappa(\bb ,\cdot)+\check{q_1}\\\check{q_2} \end{pmatrix}, \end{eqnarray}
We write the expansion (\ref{s}) with $ \partial_s \tilde q$ 
\begin{equation}\label{212}
\partial_s \tilde q=\tilde \pi_0^\bb  (\partial_s \tilde q)\tilde F_0(\bb,\cdot)+\tilde \pi_-^\bb  (\partial_s \tilde q).
\end{equation}
Using (\ref{167}) and (\ref{212}), we see that
\begin{equation}\label{naro2}
\tilde \pi_-^\bb  (\partial_s \tilde q)-\partial_s \tilde q=0.
\end{equation}
From the remark after the Definition \ref{4.6}, we see that $\tilde L_{\bb } (\tilde q) \in \mathcal{ H}_-$ (as $\tilde q\in \mathcal{ H}_-$) and
\begin{equation}\label{naro3}
\tilde \pi_-^\bb  (\tilde L_{\bb } (\tilde q))=\tilde L_{\bb } (\tilde q).
\end{equation}
Using (\ref{s}) with $(\kappa(\bb ,y),0)$, (\ref{1}) and (\ref{pi}), we get 
\begin{equation*}
\tilde \pi_-^\bb  \begin{pmatrix} \kappa(\bb ,\cdot)\\0 \end{pmatrix}=0,
\end{equation*}
therefore, using (\ref{s}) with $\check q$, we write
\begin{eqnarray}\label{naro4}
\Big |\Big |\tilde \pi_-^\bb \begin{pmatrix} \kappa(\bb ,y)+\check{q_1}\\\check{q_2} \end{pmatrix}\Big |\Big |_{\mathcal{H}}&=&
\Big |\Big |\tilde \pi_-^\bb \begin{pmatrix} \check{q_1}\\\check{q_2} \end{pmatrix}\Big |\Big |_{\mathcal{H}}
\le|| \check q ||_{\mathcal{H}}+|\tilde \pi_0^\bb  ( \check q)|||\tilde F_0(\bb,\cdot)||_{\mathcal{H}}\notag \\
&\le& C || \check q ||_{\mathcal{H}} \le C_0(\check \alpha_1^2 +\check \alpha_{-}^2+\tilde \alpha_{-}^2)^\frac{1}{2},
\end{eqnarray}
where we used (\ref{pi}), (\ref{normalization}), (\ref{majoration1}) and (\ref{183}) to get the last line.
Using (\ref{naro1}), (\ref{naro2}), (\ref{naro3}), (\ref{naro4}) and (\ref{184}) we get (\ref{201}).\\
$\bullet${\it Proof of (\ref{202})}: Note from (\ref{s}) that
\begin{equation*}
\begin{pmatrix} 0\\ \tilde f_\bb (q_1)\end{pmatrix} = \tilde B_0(s) \tilde F_0(\bb  ,\cdot)+\tilde \pi_-^\bb  \begin{pmatrix} 0\\ \tilde f_\bb (q_1)\end{pmatrix},
\end{equation*}
where $\tilde B_0(s)=\tilde \pi_0^\bb  \begin{pmatrix} 0\\ \tilde f_\bb (q_1)\end{pmatrix}$. So from the definition (\ref{134}), the bilinearity of $\tilde \varphi_\bb $ and the bound on the norm of $\tilde F_0$, we have
\begin{eqnarray*}
&\,&\Big |\tilde \varphi_\bb  \left( \tilde q, \tilde \pi_-^\bb  \begin{pmatrix} 0\\ \tilde f_\bb (q_1)\end{pmatrix}\right)-\int_{-1}^{1} \tilde q_2 \tilde f_\bb (q_1) \rho dy \Big |\notag\\
&=&\Big |\tilde \varphi_\bb  \left(\tilde q, \begin{pmatrix} 0\\ \tilde f_\bb (q_1)\end{pmatrix} -\tilde B_0(s) \tilde F_0(\bb  ,y)\right)-\tilde \varphi_\bb  \left(\tilde q, \begin{pmatrix} 0\\ \tilde f_\bb (q_1)\end{pmatrix}\right)\Big |\notag\\
&=&\Big |\tilde \varphi_\bb  \left(\tilde q, \tilde B_0(s) \tilde F_0(\bb  ,y) \right) \Big |\le  C |\tilde B_0(s)|||\tilde q||_{\mathcal{H}} \le C (\check \alpha_1^2 +\check \alpha_{-}^2+\tilde \alpha_{-}^2)^\frac{3}{2},
\end{eqnarray*}
from (\ref{198}) and (\ref{183}), (\ref{202}) follows.\\
$\bullet${\it Proof of (\ref{203})}:
We see from the expression of $\mathcal{ F}_{\bb(s)}(q_1)$ (\ref{*F*}) that 
\begin{align}\label{x2}
&\int_{-1}^{1} \check q_2  \check{f}_{\bb }(q_1) +\int_{-1}^{1} \tilde q_2  \tilde{f}_{\bb }(q_1) \rho dy= \int_{-1}^{1} \partial_s \mathcal{ F}_{\bb(s)}(q_1) \rho dy \\&+\notag\int_{-1}^{1}\frac{p-1}{2}\kappa^{p-2}(\bb ) \partial_\bb  \kappa (\bb ) \bb'(s) (p\check q_1^2+\tilde q_1^2) \rho dy
+\theta'(s) \int_{-1}^{1}\left[(\kappa(\bb ,y)+\check q_1)\tilde{f}_{\bb }(q_1)-\tilde q_1 \check{f}_{\bb }(q_1)\right] \rho dy.
\end{align}
Since we have $||\partial_\bb  \kappa(\bb ,\cdot)\kappa(\bb ,_\cdot)^{p-2}||_{L_\rho^{\frac{p+1}{p-1}}}\le C_0/(1-\bb^2)$, from the expression of $\partial_\bb  \kappa(\bb ,y)$
$$\partial_\bb  \kappa(\bb ,y)=-\frac{2 \kappa_0  }{p-1} \frac{(y+\bb )(1-\bb^2)^{\frac{2-p}{p-1}}} {(1+\bb y)^{-\frac{p+1}{p-1}}},$$ 
the definition of (\ref{defk}) $\kappa (\bb ,y)$ and Claim \ref{claim}, we use the H\"{o}lder inequality and the Hardy-Sobolev inequality of Lemma \ref{2.2} to derive that
\begin{equation*}
\left|\int_{-1}^{1} \partial_\bb  \kappa(\bb ,y)\kappa(\bb ,y)^{p-2}(p \check q_1^2+\tilde q_1^2) \rho dy \right|\le \frac{C_0}{1-\bb^2} ||q(s)||_{\mathcal{H}}^2,
\end{equation*}
so, by (\ref{183}) and (\ref{184}), we get
\begin{equation}\label{x3}
\Big|(p-1)\bb'(s) \int_{-1}^{1}\kappa^{p-1}(\bb ,y) \partial_\bb  \kappa (\bb ,y)(p\check q_1^2+\tilde q_1^2) \rho dy\Big|\le C_0 (\check \alpha_1^2 +\check \alpha_{-}^2+\tilde \alpha_{-}^2)^2.
\end{equation}
Since $\forall |\bb|<1$, $||\kappa(\bb ,\cdot)||_{\mathcal{H}_0}\le C_0$ by definition (\ref{defk}), using Claim \ref{5.3}, Lemma \ref{2.2} and (\ref{183}) we see that
\begin{align}
&\Big| \int_{-1}^{1} \check q_1 \tilde{f}_{\bb }(q_1) \rho (y) dy\Big|+\Big| \int_{-1}^{1} \tilde q_1 \check{f}_{\bb }(q_1) \rho (y) dy\Big|\notag\\ &\le C_0\left( \delta_{\{p\ge2\}}\int_{-1}^{1} \kappa(\bb ,y)^{p-2} |q_1(y,s)|^3\rho dy+  \int_{-1}^{1} |q_1(y,s)|^{p+1}\rho dy\right)\notag\\
&\le C_0 \left(\delta_{\{p\ge2\}} ||\kappa(\bb ,\cdot)||^{p-2}_{L_\rho^{p+1}}||q_1||^{p+1}_{L_\rho^{p+1}}+ ||q_1||^{p+1}_{L_\rho^{p+1}} \right)\notag\\
&\le  C_0||q_1(s)||_{\mathcal{H}}^{\check p+1},\;\mbox{ where }\check p=\min(p,2)>1 \notag\\
&\le C_0 (\check \alpha_1^2 +\check \alpha_{-}^2+\tilde \alpha_{-}^2).\label{x5}
\end{align}
Using (\ref{x2}), (\ref{x3}), (\ref{198}), (\ref{x5}) and (\ref{184}) we see that estimate (\ref{203}) holds. This concludes the proof of Lemma \ref{5.4}.
\end{proof}
\medskip
\noindent{\bf Part 3: An additional relation }\\
(iii) First, note from (\ref{183}) that
\begin{align}\label{224}
 \int_{-1}^{1} \check q_1^2 \rho+\int_{-1}^{1} (\partial_y \check q_1)^2 (1-y^2)\rho+\int_{-1}^{1} \check q_2^2 \rho&\le C_0  (\check \alpha_1^2 +\check \alpha_{-}^2),
\\\label{224'}
 \int_{-1}^{1} \tilde q_1^2 \rho+\int_{-1}^{1} (\partial_y \tilde q_1)^2 (1-y^2)\rho+\int_{-1}^{1} \tilde q_2^2 \rho&\le C_0  \tilde\alpha_{-}^2,
\end{align}
and, using equation (\ref{170}) and the definition of $\check L_\bb $ (\ref{barL_bb }), we write
\begin{eqnarray}\label{souris}
&\,&\frac{d}{ds}\int_{-1}^{1} \check q_1 \check q_2 \rho =\int_{-1}^{1}(\theta' (s) \tilde q_1 \check q_2+\check q_2^2-\bb'(s)  \check q_2\partial_\bb  \kappa (\bb ,\cdot)) \rho\notag\\&+&\int_{-1}^{1}\check q_1 (\mathcal{L} \check{q_1}+\check\psi (\bb , \cdot)\check{q_1}-\frac{p+3}{p-1} \check{q_2}- 2 y \partial_y \check{q_2}+ \check f_\bb (q_1)+\theta' (s) \tilde q_2 )\rho.
\end{eqnarray}
Almost of the terms in the right hand side of (\ref{souris}) have been studied in \cite{MR2362418}, except for the two terms with $\theta' (s)$.\\
-Using the Cauchy-Schwartz inequality, (\ref{179}), (\ref{183}) and (\ref{184}), we see that
\begin{eqnarray}\label{129.5}
\Big{|}\int_{-1}^{1} \theta' (s) \tilde q_1 \check q_2\rho \Big{|}&\le&|\theta' (s)|\left( \int_{-1}^{1} \tilde q_1^2\rho \right)^\frac{1}{2}\left( \int_{-1}^{1} \check q_2^2\rho \right)^\frac{1}{2}\le \frac{1}{100} (\check \alpha_1^2 +\check \alpha_{-}^2+\tilde \alpha_{-}^2),
\end{eqnarray}
and,
\begin{eqnarray}
\Big{|}\int_{-1}^{1} \theta' (s) \check q_1 \tilde q_2\rho \Big{|}&\le&|\theta' (s)|\left( \int_{-1}^{1} \check q_1^2\rho \right)^\frac{1}{2}\left( \int_{-1}^{1} \tilde q_2^2\rho \right)^\frac{1}{2}\le \frac{1}{100} (\check \alpha_1^2 +\check \alpha_{-}^2+\tilde \alpha_{-}^2),
\end{eqnarray}
for $\epsilon$ small enough. Using the proof of Proposition 5.2 page 103 in \cite{MR2362418} to control the other terms we get (\ref{188}).\\
By the same way, in order to prove (\ref{188'}), we will bound all the terms on the right hand side of the following:
\begin{eqnarray}\label{souris'}
&\,&\frac{d}{ds}\int_{-1}^{1} \tilde q_1 \tilde q_2 \rho =\int_{-1}^{1}(-\theta' (s) \tilde q_2 (\kappa (\bb,\cdot )+\tilde q_1)+\tilde q_2^2) \rho\notag\\&+&\int_{-1}^{1}\tilde q_1 (\mathcal{L} \tilde{q_1}+\tilde\psi (\bb , \cdot)\tilde{q_1}-\frac{p+3}{p-1} \tilde{q_2}- 2 y \partial_y \tilde{q_2}+ \tilde f_\bb (q_1)-\theta' (s) \check q_2 )\rho.
\end{eqnarray}\\
-We use the Cauchy-Schwartz inequality, (\ref{184}), (\ref{224}), (\ref{224'}), (\ref{179}) and bound $\kappa(\bb ,y)$ to write for $\epsilon$ small enough,
\begin{eqnarray}\label{m1}
\left|\int_{-1}^{1}\theta' (s) \tilde q_2 (\kappa(\bb ,y)+\tilde q_1)\rho\right |&\le& |\theta' (s)|\left( ||\kappa (\bb ,\cdot)||_{L^2_\rho}+(\int_{-1}^{1} \tilde q_1^2\rho)^\frac{1}{2}\right) (\int_{-1}^{1} \tilde q_2^2\rho)^\frac{1}{2}\notag\\
&\le& C (\check \alpha_1^2 +\check \alpha_{-}^2+\tilde \alpha_{-}^2)^\frac{3}{2}\le\frac{1}{100} (\check \alpha_1^2 +\check \alpha_{-}^2+\tilde \alpha_{-}^2).
\end{eqnarray}
-From the definition of $\tilde \varphi_\bb $ (\ref{135}) and the definition of $\tilde \alpha_-$ (\ref{181'}), we write
\begin{equation}\label{m2}
 \int_{-1}^{1}\tilde q_1 (\mathcal{L} \tilde{q_1}+\tilde\psi (\bb , \cdot)\tilde{q_1})\rho=-\tilde \varphi_\bb \begin{pmatrix} \begin{pmatrix}\tilde{q_1}\\0\end{pmatrix},\begin{pmatrix}\tilde{q_1}\\0\end{pmatrix}\end{pmatrix}=-\tilde \alpha_{-}^2.
\end{equation}
-Using integration by parts, the fact that $|y\partial_y \rho (y)|\le C \frac{\rho}{1-y^2}$, the Cauchy-Schwartz inequality, Lemma \ref{2.2}, (\ref{224}) and (\ref{224'}), we write
\begin{eqnarray}\label{m3}
 &\,&\left | -\frac{p+3}{p-1} \int_{-1}^{1} \tilde q_1 \tilde q_1  \rho -2\int_{-1}^{1} \tilde q_1 y \partial_y\tilde q_2 \rho\right |\notag\\
&=&\left | 2\int_{-1}^{1} \tilde q_2  \partial_y\tilde q_1  y\rho+\left(2-\frac{p+3}{p-1} \right)\int_{-1}^{1} \tilde q_1 \tilde q_1  \rho+2\int_{-1}^{1} \tilde q_2 \tilde q_1y \partial_y \rho \right |\notag\\
&\le& C\int_{-1}^{1}\left( |\tilde q_2||\partial_y \tilde q_1|\rho+|\tilde q_2||\tilde q_1| \frac{\rho}{1-y^2}\right)\notag\\
&\le&\left(\int_{-1}^{1} \tilde q_2^2 \frac{\rho}{1-y^2}\right)^{1/2} \left[ \int_{-1}^{1} ( \partial_y \tilde q_1)^2(1-y^2)\rho+\int_{-1}^{1} \tilde q_1^2 \frac{\rho}{1-y^2}\right]^{1/2}\\
&\le& C_0 \tilde \alpha_{-} (\int_{-1}^{1} \tilde q_2^2 \frac{\rho}{1-y^2})^{1/2}\le\frac{1}{100}\tilde \alpha_{-}^2+C\int_{-1}^{1} \tilde q_2^2 \frac{\rho}{1-y^2}.\notag
\end{eqnarray}
- Arguing as for (\ref{x5}) and (\ref{129.5}), using (\ref{179}) we write for $\epsilon$ small enough 
\begin{align}\label{m4}
\Big| \int_{-1}^{1} \tilde q_1 \tilde{f}_{\bb }(q_1) \rho (y) dy\Big| &\le C_0||q_1(s)||_{\mathcal{H}}^{\check p+1}\le \frac{1}{100}(\check \alpha_1^2 +\check \alpha_{-}^2+\tilde \alpha_{-}^2)\\
% We use the Cauchy-Schwartz inequality, (\ref{184}), (\ref{224}), (\ref{224'}) and (\ref{179}) to write for $\epsilon$ small enough,
\label{m5}
\left | \int_{-1}^{1}\theta' (s) \check q_2 \tilde q_1\rho\right | &\le C|\theta' (s)| (\check \alpha_1^2 +\check \alpha_{-}^2+\tilde \alpha_{-}^2)\le\frac{1}{100} (\check \alpha_1^2 +\check \alpha_{-}^2+\tilde \alpha_{-}^2).
\end{align}
Collecting (\ref{souris'})-(\ref{m5}), we get
\begin{equation*}
 \frac{d}{ds}\int_{-1}^{1} \tilde q_1 \tilde q_2 \rho \le -\frac{4}{5}\tilde \alpha_{-}^2+ C_0\int_{-1}^{1} \tilde q_2^2\frac{\rho}{1-y^2}+ C_0 (\check \alpha_1^2 +\check \alpha_{-}^2).
\end{equation*}

\bigskip 
{\bf Part 4: Energy barrier}\\
(iv) Using the definition of $q(y,s)$ (\ref{168}), we can make an expansion of $E(w(s),\partial_s w(s))$ (\ref{15}) for $q\rightarrow 0$ in $\mathcal{H}$ and get after from straightforward computations
\begin{equation}
 E(w(s),\partial_s w(s))= E(\kappa_0,0)+\frac{1}{2}(\check \varphi_\bb (\check q, \check q)+\tilde \varphi_\bb (\tilde q,\tilde q))-\int_{-1}^{1}\mathcal{F}_\bb  (q_1) \rho dy \label{228}
\end{equation}
where $ \check \varphi_\bb $, $ \tilde \varphi_\bb $ and $\mathcal{F}_\bb  (q_1)$ are defined in (\ref{134'}), (\ref{134}) and (\ref{*F*}).\\
Since we have (\ref{209}), (\ref{187}), (\ref{179}) and (\ref{183}):
\begin{equation}
 \left | \int_{-1}^{1}\mathcal{F}_\bb  (q_1) \rho dy \right | \le C ||q(s)||_{\mathcal{H}}^{\check p+1}\le C \epsilon^{\check p-1} (\check \alpha_1^2+\check \alpha_-^2+\tilde \alpha_-^2),\label{229}
\end{equation}
we note that for some $C_1 >0$
\begin{equation}\label{230}
 \check \varphi_\bb (\check q, \check q) \le C_0 \check \alpha_1^2- C_1\check \alpha_-^2,
\end{equation}
which was proved in \cite{MR2362418} in page 113. From (\ref{17}), (\ref{228}), (\ref{230}) and (\ref{229}), we see that taking $\epsilon$ small enough so that $C \epsilon^{\check p-1}<\frac{C_1}{4}$, we get
\begin{equation*}
 0\le E(w(s),\partial_s w(s))-E(\kappa_0,0)\le \left( \frac{C_0}{2}+\frac{C_1}{4}\right) \check \alpha_-^2-\frac{C_1}{4}\check \alpha_1^2+\left(\frac{1}{2}+\frac{C_1}{4}\right) \tilde \alpha_-^2.
\end{equation*}
which yields (\ref{189}).
\end{proof}
\subsection{Exponential decay of the different components}\label{4.4}
Our aim is to show that $||q(s)||_{\mathcal{H}}\rightarrow 0$ and that both $\theta$ and $\bb$ converge as $s\rightarrow \infty$. An important issue will be to show that the unstable mode $\check \alpha_1$, which satisfies equation (\ref{183}) never dominates. This is true thanks to item $(iv)$ in Proposition \ref{5.2} (see the third remark following that Proposition for more details). Let us first introduce a more adapted notation and rewrite Proposition \ref{5.2}.\\
If we introduce
\begin{equation}\label{234}
 \lambda(s)=\frac{1}{2} \log\left(\frac{1+\bb(s)}{1-\bb(s)}\right), a(s)=\check \alpha_1(s)^2\, \mbox{and}\, b(s)=\check\alpha_-(s)^2+\tilde\alpha_-(s)^2+R_-(s) 
\end{equation}
(note that $\bb(s)=\tanh(\lambda(s))$), then we see from (\ref{187}), and (\ref{183}) that if (\ref{179}) holds, then $|b-(\check\alpha_-(s)^2+\tilde\alpha_-(s)^2)|\le C_0 \epsilon^{\check p-1}(\check\alpha_1(s)^2+\check\alpha_-(s)^2+\tilde\alpha_-(s)^2)$, hence
\begin{equation}\label{235}
 \frac{99}{100}\check\alpha_-(s)^2+ \frac{99}{100}\tilde\alpha_-(s)^2-\frac{1}{100} a\le b\le  \frac{101}{100}\check\alpha_-(s)^2+ \frac{101}{100}\tilde\alpha_-(s)^2+\frac{1}{100} a
\end{equation}
for $\epsilon$ small enough. Therefore, using Proposition \ref{5.2}, estimate (\ref{179}), (\ref{183}) and the fact that $\lambda'(s)=\frac{\bb'(s)}{1-\bb(s)^2}$, we derive the following:
\begin{cl}\label{5.5}{\bf (Relations between $a$, $b$, $\lambda$, $\theta$, $\int_{-1}^1 \check q_1\check q_2 \rho$ and $\int_{-1}^1 \tilde q_1\tilde q_2 \rho$)} There exist positive $\epsilon_4$, $K_4$ and $K_5$ such that if w is a solution to equation (\ref{equa}) such that (\ref{167}) and (\ref{179}) hold at some time s for some $\epsilon\le \epsilon_4$, where q is defined in (\ref{168}), then using the notation (\ref{234}), we have:
\item{(i)} {\bf (Size of the solution)}
\begin{align}
\frac{1}{K_4}(a(s)+b(s))\le || q(s)||_\mathcal{H}^2&\le K_4 (a(s)+b(s))\le K_4^2 \epsilon^2,\label{236}
\\
 |\theta'(s)|+|\lambda'(s)|&\le K_4 (a(s)+b(s))\le K_4^2 ||q(s)||_\mathcal{H}^2 ,\label{237}
\\
 \left | \int_{-1}^{1}\check q_1\check q_1\rho \right |&\le K_4 (a(s)+b(s)),\label{238}
\\
\left | \int_{-1}^{1}\tilde q_1\tilde q_1\rho \right |&\le K_4 b(s),\label{238'}
 \end{align}
and (\ref{235}) holds.
\item{(ii)} {\bf (Equations)}
\begin{align}\label{239}
 \frac{3}{2} a-K_4 \epsilon b&\le a' \le \frac{5}{2} a-K_4 \epsilon b,
\\\label{240}
b'&\le -\frac{8}{p-1}\int_{-1}^{1}(\check q_{-,2}^2+\tilde q_{-,2}^2)\frac{\rho}{1-y^2}dy+ K_4\epsilon (a+b),
\\\notag
 \frac{d}{ds}\int_{-1}^1 (\check q_1 \check q_2+ \tilde q_1 \tilde q_2) \rho &\le -\frac{3}{5} b+K_4 \int_{-1}^1 (\check q_{-,2}^2+\tilde q_{2}^2)\frac{\rho}{1-y^2}+K_4 a.
  \end{align}
\item{(iii)} {\bf (Energy barrier)} If (\ref{17}) holds, then
   \begin{equation*}
    a(s)\le K_5 b(s).
   \end{equation*}
\end{cl}

{\it Proof of Theorem \ref{theo3}:} Consider $w\in C([s^*,\infty),\mathcal{H})$ for some $s^*\in \mathbb{R}$ a solution of equation (\ref{equa}) such that (\ref{17}) and (\ref{18}) hold for some $\bb^*\in (-1,1), \theta^* \in [0, 2\pi)$ and $\epsilon^* \in (0,\epsilon_0]$. Consider then 
\begin{equation}\label{148.5}
 \epsilon=2K_0K_1 \epsilon^*
\end{equation}
where $K_1$ is given in Proposition \ref{5.1} and $K_0$ will be fixed later. If
   \begin{equation*}
     \epsilon^*\le \epsilon_1 \mbox{ and } \epsilon \le \epsilon_4,
   \end{equation*}
then we see that Proposition \ref{5.1} Corollary \ref{5.5} and (\ref{235}) apply respectively with $\epsilon^*$ and $\epsilon$. In particular, there is a maximal solution $\bb(s)\in C^1([s^*,\infty),(-1,1))$ such that (\ref{167}) holds for all $s\in[s^*,\infty)$ where $q(y,s)$ is defined in (\ref{168}) and 
\begin{equation}\label{244}
 |\theta-\theta^*|+|\lambda(s^*)-\lambda^* |+||q(s^*)||_{\mathcal{H}}\le K_1 \epsilon^* \mbox{ with }\lambda^*=\log\left(\frac{1+\bb^*}{1-\bb^*}\right).
\end{equation}
If in addition we have 
  \begin{equation}\label{245}
   K_0\ge 1 \mbox{ hence, } \epsilon\ge 2 K_1 \epsilon^*,
  \end{equation}
then, we can give two definitions:\\
- We define first from (\ref{244}) and  (\ref{245}) $s^*_1\in (s^*,\infty)$ such that for all $s\in [s^*,s_1^*],$ 
    \begin{equation}
   || q(s)||_\mathcal{H}<\epsilon\label{246}
    \end{equation}
and if $s^*_1<\infty$, then $|| q(s^*_1)||_\mathcal{H}=\epsilon$.\\
-Then, we define $s_2^* \in [s^*,s_1^*]$ as the first $s \in [s^*,s_1^*]$ such that
      \begin{equation}\label{247}
       a(s)\ge \frac{b(s)}{5 K_4}
      \end{equation}
where $K_4$ is introduced in Corollary \ref{5.5}, or $s^*_2=s^*_1$ if (\ref{247}) is never satisfied on $[s^*,s_1^*]$. We claim the following:
\begin{cl}\label{5.6}
 There exist positive $\epsilon_6$, $\mu_6$, $K_6$ and $f\in C^1([s^*, s^*_2]$ such that if $\epsilon\le \epsilon_6$, then for all $s\in [s^*,s_2^*]$: \\
(i)
\begin{equation*}
 \frac{1}{2}f(s)\le b(s)\le 2 f(s)\mbox{ and }f'(s)\le -2\mu_6f(s),
\end{equation*}
(ii)
\begin{equation*}
|| q(s)||_\mathcal{H}\le K_6 || q(s^*)||_\mathcal{H}e^{-\mu_6(s-s^*)}\le K_6 K_1 \epsilon^*e^{-\mu_6(s-s^*)}.
\end{equation*}
\end{cl}
\begin{proof} The proof of Claim 5.6 page 115 in \cite{MR2362418} remains valid where $f(s)$ is given by 
 $$f(s)=b(s)+\eta_6\int_{-1}^1 (\check q_1 \check q_2+ \tilde q_1 \tilde q_2) \rho,$$
where $\eta_6 >0$ is fixed small independent of $\epsilon$. 
\end{proof}
\begin{cl}\label{5.7}
 (i) There exists $\epsilon_7>0$ such that for all $\sigma>0$, there exists $K_7(\sigma)>0$ such that if $\epsilon \le \epsilon_7$, then
\begin{equation*}
 \forall s \in [s_2^*, \min (s_2^*+\sigma, s_1^*)],\,|| q(s)||_\mathcal{H}\le K_7 || q(s^*)||_\mathcal{H}e^{-\mu_6(s-s^*)}\le K_7 K_1 \epsilon^*e^{-\mu_6(s-s^*)}
\end{equation*}
where $\mu_6$ has been introduced in Claim \ref{5.6}.\\
(ii) There exists $\epsilon_8>0$ such that if $\epsilon\le \epsilon_8$, then
\begin{equation}\label{254}
  \forall s \in (s_2^*,  s_1^*],\; b(s)\le a(s) \left( 5 K_4 e^{-\frac{(s-s_2^*)}{2}}+\frac{1}{4 K_5}\right)
\end{equation}
where $K_4$ and $K_5$ have been introduced in Corollary \ref{5.5}.
\end{cl}
\begin{proof}
 The proof is the same as the proof of Claim 5.7 page 117 in \cite{MR2362418}.
\end{proof}
Now, in order to conclude the proof of Theorem \ref{theo3}, we fix $\sigma_0>0$ such that 
\begin{equation*}
 5K_4^{-\frac{\sigma_0}{2}}+\frac{1}{4K_5}\le \frac{1}{2K_5},
\end{equation*}
where $K_4$ and $K_5$ are introduced in Claim \ref{5.5}. Then, we impose the condition
\begin{equation}\label{259}
 \epsilon=2K_0K_1 \epsilon^*, \mbox{ where }K_0=\max(2,K_6,K_7(\sigma_0)),
\end{equation}
and the constants are defined in Proposition \ref{5.1} and Claims \ref{5.6} and \ref{5.7}. Then, we fix
\begin{equation*}
 \epsilon_0=\min \left(1,\epsilon_1,\frac{\epsilon_i}{2K_0K_1}\mbox{ for }i\in\{4,6,7,8\}\right)
\end{equation*}
and the constants are defined in Proposition \ref{5.1}, Claims \ref{5.5}, \ref{5.6} and \ref{5.7}.
Now, if $\epsilon^*\le \epsilon_0$, then Claim \ref{5.5}, Claim \ref{5.6} and Claim \ref{5.7} apply. We claim that for all $s\in[s^*,s_1^*]$,
 \begin{eqnarray}\label{261}
 || q(s)||_\mathcal{H}\le K_0 || q(s^*)||_\mathcal{H}e^{-\mu_6(s-s^*)}\le K_0 K_1 \epsilon^*e^{-\mu_6(s-s^*)}=\frac{\epsilon}{2}e^{-\mu_6(s-s^*)}.
 \end{eqnarray}
Indeed, if $s\in[s^*,\min(s_2^*+\sigma_0,s_1^*)]$, then, this comes from $(ii)$ of Claim \ref{5.6} or $(i)$ of Claim \ref{5.7} and the definition of $K_0$ (\ref{259}).\\
Now, if $s_2^*+\sigma_0<s_1^*$ and $s\in[s_2^*+\sigma_0,s_1^*]$, then we have from (\ref{254}) and the definition of $\sigma_0$, $b(s)\le\frac{a(s)}{2K_5 }$ on the one hand. On the other hand, from $(iii)$ in Claim \ref{5.5}, we have $a(s)\le K_5 b(s)$, hence, $a(s)=b(s)=0$ and from (\ref{236}), $q(y,s)\equiv 0$, hence (\ref{261}) is satisfied trivially.\\
In particular, we have for all $s\in [s ^*,s_1^*],\, ||q||_\mathcal{H}\le \frac{\epsilon}{2}$, hence, by definition of $s_1^*$, this means that $s_1^*=\infty$.\\
From $(i)$ of Claim \ref{5.7} and (\ref{237}), we have
\begin{equation}\label{262}
 \forall s \ge s^*,||q(s)||_\mathcal{H}\le \frac{\epsilon}{2} e^{-\mu_6(s-s^*)}\mbox{ and }|\theta'(s)|+|\lambda'(s)|\le K_4^2 \frac{\epsilon^2}{4} e^{-2\mu_6(s-s^*)}.
\end{equation}
Hence, there is $\theta_\infty,\,\lambda_\infty$ in $\mathbb{R}$ such that $\theta(s) \rightarrow \theta_\infty$, $\lambda(s) \rightarrow \lambda_\infty$ as $s\rightarrow \infty$ and
\begin{equation}\label{263}
  \forall s \ge s^*,| \lambda_\infty-\lambda(s) |\le C_1\epsilon^{*2} e^{-2\mu_6(s-s^*)}=C_2\epsilon^{2} e^{-2\mu_6(s-s^*)}
\end{equation}
\begin{equation}\label{263'}
  \forall s \ge s^*,| \theta_\infty-\theta(s) |\le C_1\epsilon^{*2} e^{-2\mu_6(s-s^*)}=C_2\epsilon^{2} e^{-2\mu_6(s-s^*)}
\end{equation}
for some positive $C_1$ and $C_2$. Taking $s=s^*$ here, and using (\ref{244}) and (\ref{148.5}), we see that $| \lambda_\infty-\lambda^* |+| \theta_\infty-\theta^* |\le C_0\epsilon^* $. If $\bb_\infty=\tanh \lambda_\infty,$ then we see that $|\bb_\infty-\bb^*|\le C_3 (1-\bb^{*2})\epsilon^*.$\\
Using the definition of $q$ (\ref{168}), (\ref{174}), (\ref{262}), (\ref{263}) and (\ref{263'}) %and Lemma \ref{lemme}
we write
\begin{eqnarray*}
&\,& \Bigg|\Bigg|\begin{pmatrix}w(s)\\\partial_s w(s)\end{pmatrix}-e^{i\theta_\infty}\begin{pmatrix}\kappa(\bb _\infty,\cdot)\\0\end{pmatrix}\Bigg|\Bigg|_\mathcal{H}\\
&\le& \Bigg|\Bigg|\begin{pmatrix}w(s)\\\partial_s w(s)\end{pmatrix}-e^{i\theta_\infty}\begin{pmatrix}\kappa(\bb (s),\cdot)\\0\end{pmatrix}\Bigg|\Bigg|_\mathcal{H}+||e^{i\theta(s)}(\kappa(\bb (s),\cdot)-\kappa(\bb _\infty,\cdot))||_{\mathcal{H}_0}\\&+&||\kappa(\bb _\infty,\cdot)||_{\mathcal{H}_0} |e^{i\theta(s)}-e^{i\theta(\infty)}|\\
&\le& ||q(s)||_\mathcal{H}+C|\lambda_\infty-\lambda(s)|+C|\theta_\infty-\theta (s)|\le C_4 \epsilon^*e^{-\mu_6(s-s^*)}.
\end{eqnarray*}
This concludes the proof of Theorem \ref{theo3}.
\medskip\\
\appendix 

\section{Energy estimates in similarity variables}\label{__}
For the sake of completeness, we give in this section sketches of the proofs of Proposition \ref{Th} and \ref{2} proved in the real case in \cite{MR2147056} and \cite{MR2362418}, and which extend to the complex case straightforwardly.\\
\medskip

{\bf Sketch of the proof of Proposition \ref{Th}}\\
Consider $x_0 \in \mathbb{R}$. If $T(x_0)=\min_{x\in\mathbb{R}} T(x)$, the proof is given in \cite{MZ05}. If not, then we has to use a geometrical covering argument in addition to the ideas of \cite{MZ05}. In order to keep this sketch in a reasonable length, we don't mention this covering argument and refer the reader to \cite{MR2147056}. Thus, we only focus on the real case where
$$T_0\equiv T(x_0)=\min_{x\in\mathbb{R}} T(x).$$
$\bullet${\it The lower bound:}\\
Note first that the lower bound follows from the finite speed of propagation and scaling. Indeed, if $(w,\partial_s w)$ is small in $H^1\times L^2$ at some time $s=s_0$, then using back the similarity variables transformation (\ref{trans_auto}), we see that initial data for $(u,\partial_t u)$ is small on the basis of the light cone, which means that the solution cannot blow-up at time $T$. See Remark after Theorem 1 page 1149 in \cite{MZ05}.\\
$\bullet${\it The upper bound:}\\
The proof is performed in similarity variables and relies on two arguments:\\
- The fact that the functional $E(w)$ defined in (\ref{15}) is a Lyapunov functional for equation (\ref{equa}) which satisfies
$$\frac{d}{ds}E(w(s))=-\frac{4}{p-1}\int_{-1}^1|\partial_sw|^2\frac{\rho}{1-y^2}dy.$$
- A blow-up criterion from Antonini and Merle \cite{MR1861514} stating that a solution $w$ of equation cannot be defined for all $s\in[s_0,+\infty)$ if $E(w(s_0))<0$.\\
From these two facts, we see that for any $\bar x \in \mathbb{R}$, $w_{\bar x}$ satisfies
\begin{equation} 
\forall s\ge -\log T_0,\left\{
\begin{array}{l}
\displaystyle 0 \le E(w_{\bar x}(s))\le C_0, \\
\int_{-log T_0}^{+\infty}\int_{|y|<1}|\partial_s w|^2 \frac{\rho}{1-y^2}\le C_0,
\end{array}\label{app}
\right.
\end{equation}
with this identity, the proof is done in 3 steps:\\
- {\bf Step 1}: multiplying (\ref{equa}) by $\bar w \rho$ and integrating for $x\in (-1,1)$, we obtain a new identity. Combining that identity with (\ref{app}) we end-up by proving that 
\begin{equation}\label{appen}
 \forall \bar x \in \mathbb{R},\forall s\le -\log T_0+1,\, \int_{s}^{s+1} \int_{|y|<\frac{1}{2}}\left(|\partial_s w_{\bar x}|^2+|\partial_y w_{\bar x}|^2(1-|y|^2)+|w_{\bar x}|^{p+1}\right) dy \le C_0.
\end{equation}
Note that it is important to get (\ref{app}) and (\ref{appen}) for any $\bar x \in \mathbb{R}$ and not just for $\bar x= x_0$.\\
- {\bf Step 2}: Using interpolation, Sobolev embeddings and a covering argument we end-up with the fact  that
\begin{equation}\label{appendi}
 \forall \bar x \in \mathbb{R},\forall s\ge -\log T_0+1,\,  \int_{|y|<1}|w_{\bar x}|^{p+1} dy \le C_0.
\end{equation}
- {\bf Step 3}: Given $s\ge -\log T_0+1$, we need to work at $\check x=\check x(s)$ such that
\begin{equation}\label{appendix}
 \int_{|y|<1}|\nabla w_{\check x (s)}(y,s)|^2 dy\ge\frac{1}{2}\sup_{x\in \mathbb{R}}\int_{|y|<1}|\nabla w_{ x (s)}(y,s)|^2 dy.
\end{equation}
Thanks to a covering technique, we see that at such an $\bar x (s)$, we have the equivalence of pure and weighted $L^2$ norm of the gradient, in the sense that
\begin{equation*}
 \frac{1}{C_0} \int_{|y|<1}|\nabla w_{\check x (s)}(y,s)|^2 dy\le \int_{|y|<1}|\nabla w_{\check x (s)}(y,s)|^2 \rho (1-y^2)dy\le C_0  \int_{|y|<1}|\nabla w_{\check x (s)}(y,s)|^2 dy,
\end{equation*}
this equivalence of norms is really crucial to finish the proof. Then we need the following 
% consequence of (\ref{appendi})
and Gagliardo-Nirenberg inequality (see Proposition 3.2 page 1158 in \cite{MZ05})
\begin{equation}\label{GNB}
 \int_{|y|<1}| w|^{P+1} dy\le C\left( \int_{|y|<1}(\partial_y w)^2 dy\right)^\beta,\mbox{ for some }\beta<1.
\end{equation}
Indeed, thanks to interpolation estimates in Sobolev spaces and the Gagliardo-Nirenberg inequality (\ref{GNB}), we use the energy boundedness (\ref{app}) to derive that
$$ \forall s\ge -\log T_0+1,\; \int_{|y|<1}|\nabla w_{\check x (s)}(y,s)|^2  dy\le K(C_0)$$
for some $K>0.$ Using (\ref{appendix}), we see that
$$ \forall s\ge -\log T_0+1,\forall \check  x \in \mathbb{R},\; \int_{|y|<1}|\nabla w_{\check x (s)}(y,s)|^2  dy\le 2 K(C_0).$$
Using the definition of the functional $E(w)$ (\ref{15}) and a covering argument together with (\ref{appendi}), we conclude the proof of the upper bound of Proposition \ref{Th}. For details, see \cite{MZ05} and \cite{MR2147056}.\\
\medskip

{\bf Sketch of the proof of Proposition \ref{2}}\\
The idea is simple: equation (\ref{equa}) has a Lyapunov functional $E(w)$ defined in (\ref{15}) which satisfies
\begin{equation}
 \forall s\ge -\log T(x_0),\,\frac{d}{ds}E(w_{x_0}(s))=-\frac{4}{p-1}\int_{-1}^1(\partial_sw_{x_0})^2\frac{\rho}{1-y^2}dy.
\end{equation}
From (\ref{app}), it follows that
\begin{equation}
\int_{-\log T(x_0)}^{+\infty}\int_{|y|<1}(\partial_sw)^2\frac{\rho}{1-y^2}dy ds\le E(w_{x_0}(-\log T(x_0))\le C_0.
\end{equation}
From this identity, we see that for any $\epsilon >0$
$$\int_{s}^{s+1}\int_{|y|<1-\epsilon}(\partial_sw)^2\frac{\rho}{1-y^2}dy ds\rightarrow 0 \mbox{ as }s\rightarrow +\infty.$$
This means that $\partial_s w \rightarrow 0$, in a certain sense, which means that $w$ would approach the set of stationary solutions. From the lower bound in Proposition \ref{Th}, $w$ cannot approach the zero solutions. Since the set of non zero solutions of (\ref{equa}) is given by $ \{0, e^{i\theta} \kappa (\bb ,\cdot), |\bb|<1, \theta \in \mathbb{R}\},$ we get the conclusion. For the actual proof and for details, see Theorem 2 page 47 in \cite{MR2362418}.

\section{Explicit solution of the elliptic equation in one space dimension in $H^1$}\label{appendix1}
We solve here equation (\ref{po}), deriving the well-known KDV solutions. Our aim is to prove (\ref{essou}). Multiplying equation (\ref{po}) by $W'$ and integrating in space we have
\begin{equation}
 \frac{1}{2}W'^2-\frac{c}{2}W^2+\frac{|W|^{p+1}}{p+1}=K. \label{01}
\end{equation}
As $W\in H^1$, all the terms of the left hand side of (\ref{01}) are integrable, so $K=0$. We claim that
\begin{equation}
 \exists \xi_0 \in \mathbb{R}, W'(\xi_0)= 0,\label{06}
\end{equation}
otherwise, if
\begin{equation}\label{07}
 \forall \xi \in \mathbb{R}, W'(\xi)\neq 0,
\end{equation}
$W$ would be monotonic, with limits (in $\mathbb{\bar R}$) at $\pm$ infinity. Since $W\in L^2$, those limits have to be zero, leading to $W\equiv 0$, contradicting (\ref{07}). Thus (\ref{06}) holds, and from (\ref{01}), we see that either
\begin{equation}
 W(\xi_0)=0 \mbox{ or }W(\xi_0)=\pm \kappa_0, \label {066}
\end{equation}
 given in (\ref{defk}). Since we already know two solutions satisfying (\ref{06}) and (\ref{066}), namely
$$W\equiv 0 \mbox{ or }W(\xi)=\check k (\xi-\xi_0),$$
this concludes the proof of (\ref{essou}).
\section{Basic properties and some results}
In the following, we recall some results which we have used in this work. We first give the boundedness for $E$.
\begin{prop}\label{2.1}{\bf (Boundedness of the Lyapunov functional for equation (\ref{equa}))}
Consider $w(y,s)$ a solution to (\ref{equa}) defined for all $(y,s)\in (-1,1)\times [-\log T, +\infty)$ such that\\ $(w,\partial_s w)(-\log T) \in H^1 \times L^2(-1,1).$ For all $s\ge -\log T,$ we have
$$0\le E(w(s),\partial_s w(s)) \le E(w(-\log T),\partial_s w(-\log T)),$$
and
$$\int_{-\log T}^\infty \int_{-1}^1 (\partial_s w(y,s))^2 \frac{\rho (y)}{1-y^2} dy ds \le \frac{p-1}{4} E(w(-\log T),\partial_s w(-\log T)).$$
\end{prop}
\begin{proof}
 See Antonini and Merle \cite{MR1861514}.
\end{proof}
These following properties have been cited and proved in Section 2 in \cite{MR2362418}. We first give Hardy-Sobolev identities in the space $\mathcal{H}_0$ (\ref{10}).
 \begin{lem} \label{2.2} {\bf (A Hardy-Sobolev type identity)} For all $h\in \mathcal{H}_0$, it holds that
  \begin{equation*}
  \left( \int_{-1}^{1} h(y)^2\frac{\rho(y)}{1-y^2} dy\right)^{1/2}+||h||_{L_\rho^{p+1}}+||h(1-y^2)^\frac{1}{p-1}||_{L^{\infty}(-1,1)} \le C ||h||_{\mathcal{H}_0}.
  \end{equation*}
  \end{lem}
The operator $\mathcal{L}$ introduced in equation (\ref{equa}) have the following properties:
\begin{prop}\label{L}{\bf (Properties of the operator $\mathcal{L}$ (\ref{8}))} The operator $\mathcal{L}$ is self-adjoint in $L_\rho ^2$. For each $n\in \mathbb{N} $, there exists a polynomial $h_n$ of degree n such that
\begin{equation}
 \mathcal{L}h_n= \gamma_n h_n \;\mbox{where} \;\gamma_n=-n\left( n+\frac{p+3}{p-1}\right).\label{26}
\end{equation} 
The family $\{h_n | n\in \mathbb{N}\}$ is orthogonal and spans the whole space $L_\rho ^2$. When $n=0$ and $n=1$, the eigenfunctions are $h_0=c_0$ and $h_1=c_1 y$ for some positive $c_0$ and $c_1$, and
 \begin{equation*}
 \mathcal{L}c_0= 0,  \, \mathcal{L}c_1 y= -\frac{2(p+1)}{p-1} c_1 y.
\end{equation*} 
\end{prop}
We claim also the following:
\begin{lem}\label{too}
Consider $u\in L_\rho ^2$ such that $\mathcal{L} u \in L_\rho ^2$ and 
\begin{equation*}
 \int_{-1}^{1} u(y) \rho (y) dy=0
\end{equation*}
Then, $\int_{-1}^{1} u\mathcal{L} u\rho dy \le \gamma_1 \int u^2 \rho dy$ where $\gamma_1=-2\frac{p+1}{p-1}$.
\end{lem}
Using the Lorentz transform we get a one dimensional group which keeps invariant equation (\ref{equa}):
\begin{lem}\label{2.6}{\bf (The Lorentz transform in similarity variables)} Consider $w(y,s)$ a solution of equation (\ref{equa}) defined for all $|y|<1$ and $s\in (s_0, s_1)$ for some $s_0$ and $s_1$ in $\mathbb{R}$, and introduce for any $\bb \in (-1,1),$ the function $W\equiv \mathcal{T}_\bb  (w)$ defined by
 \begin{equation}\label{transformation}
W(Y,S)=\frac{(1-\bb^2)^\frac{1}{p-1}}{(1+\bb Y)^\frac{2}{p-1}} w(y,s), \mbox{ where }
y=\frac{Y+\bb }{1+\bb Y} \mbox{ and }s=S-\log\frac{1+\bb Y}{\sqrt{1-\bb^2}}. 
\end{equation}
Then $W(Y,S)=\mathcal{T}_\bb  (w)$ is also a solution of (\ref{equa}) defined 
$$\mbox{for all }|Y|<1\mbox{ and }S\in \left(s_0+\frac{1}{2}\log\frac{1+|\bb|}{1-|\bb|}, s_1-\frac{1}{2}\log\frac{1+|\bb|}{1-|\bb|}\right).$$
\end{lem}
In the following, we recall from \cite{MR2362418} some properties of the transformation $\mathcal{T}_\bb $:
\begin{lem}{\bf (Continuity of $\mathcal{T}_\bb $)} \label{ff}There exists $C_0>0$ such that for all $\bb \in (-1,1)$ and $v\in \mathcal{H}_0,$ we have
\item{(i)}{(Continuity of $\mathcal{T}_\bb $ in $\mathcal{H}_0$)}
\begin{eqnarray*}
\frac{1}{C_0} ||v||_{\mathcal{H}_0}\leq ||\mathcal{T}_\bb  (v)||_{\mathcal{H}_0}\leq C_0 ||v||_{\mathcal{H}_0}.
\end{eqnarray*}
\item{(ii)} For any $V_1$ and $V_2$ in $L^2_\rho$, we have
$$ \int_{-1}^{1} V_1(y) V_2(y) \rho (y) dy=\int_{-1}^{1} \frac{1-\bb^2}{(1-\bb z)^2}v_1(z) v_2(z) \rho (z) dz$$
where $v_i=\mathcal{T}_{-\bb} V_i$, $i\in\{1,2\}$.
\end{lem}
\begin{proof}
$(i)$ see Lemma $2.8$ page 57 in \cite{MR2362418}.\\
$(ii)$ This is a direct consequence of the change of variable (\ref{transformation}) and the definition (\ref{8}) of $\rho(y)$.
\end{proof}
\noindent {\bf Acknowledgments}: The author would like to thank the referee for his remarks and suggestions which undoubtedly greatly improved the presentation of our results. The author would like to thank also Professor H. Zaag for useful suggestions during the preparation of this paper.

% \bibliographystyle{plain}
%  \bibliography{references}
\medskip
\noindent{\bf Address:}\\
Universit\'e Paris 13, Institut Galil\'ee, Laboratoire Analyse G\'eometrie et Applications, \\
  CNRS-UMR 7539, 99 avenue J.B. Cl\'ement 93430, Villetaneuse, France.
\\ \texttt{e-mail: azaiez@math.univ-paris13.fr}
\end{document}